\documentclass[11pt]{article}
\usepackage[centering,height=220mm,width=150mm]{geometry}

\usepackage{amsmath,amsthm,amssymb,amsfonts,mathrsfs}
\usepackage{xcolor}
\allowdisplaybreaks

\numberwithin{equation}{section}
\theoremstyle{plain}
\newtheorem{theorem}{Theorem}[section]

\newtheorem{proposition}[theorem]{Proposition}

\newtheorem{lemma}[theorem]{Lemma}

\theoremstyle{definition}

\newtheorem{definition}[theorem]{Definition}

\newcommand{\Dw}{\mathscr{D}^{\omega}}
\newcommand{\Db}{\mathscr{D}^{\beta}}
\newcommand{\dist}[0]{\operatorname{dist}}
\newcommand{\pair}[2]{\langle #1,#2 \rangle}
\newcommand{\ave}[1]{\langle #1\rangle}
\newcommand{\bve}[1]{\big\langle #1\big\rangle}
\newcommand{\bair}[2]{\big\langle #1,#2 \big\rangle}

\newcommand{\BMO}[0]{\operatorname{BMO}}
\newcommand{\sigmaout}{\sigma_{\operatorname{out}}}
\newcommand{\sigmanear}{\sigma_{\operatorname{near}}}
\newcommand{\sigmadeep}{\sigma_{\operatorname{deep \ in}}}
\newcommand{\sigmashallow}{\sigma_{\operatorname{shallow \ in}}}
\newcommand{\pro}{\operatorname{prod}}

\begin{document}

\title{Bloom Type Inequality: The Off-diagonal Case\thanks{This work was partially supported by the
National Natural Science Foundation of China (11525104, 11531013, 11761131002 and 11801282).}}

\author{Junren Pan \quad and  \quad  Wenchang Sun\\
\small
School of Mathematical Sciences and LPMC,  Nankai University,
      Tianjin~300071, China \\
 \small   Emails: 1106840509@qq.com, \,\, sunwch@nankai.edu.cn}

\date{}
\maketitle

\begin{abstract}
In this paper, we establish a representation formula for fractional integrals. As a consequence, for two fractional integral operators $I_{\lambda_1}$ and
 $I_{\lambda_2}$, we prove a Bloom type inequality
\begin{align*}
\mbox{\hbox to 8em{}}& \hskip -8em \left\|\big[I_{\lambda_1}^1,\big[b,I_{\lambda_2}^2\big]\big]
   \right\|_{L^{p_2}(L^{p_1})(\mu_2^{p_2}\times\mu_1^{p_1})\rightarrow L^{q_2}(L^{q_1})(\sigma_2^{q_2}\times\sigma_1^{q_1})}
   \lesssim_{\substack{[\mu_1]_{A_{p_1,q_1}(\mathbb R^n)},[\mu_2]_{A_{p_2,q_2}(\mathbb R^m)} \\ [\sigma_1]_{A_{p_1,q_1}(\mathbb R^n)},[\sigma_2]_{A_{p_2,q_2}(\mathbb R^m)}}}
   \|b\|_{\BMO_{\pro}(\nu)},
\end{align*}
where the indices satisfy $1<p_1<q_1<\infty$, $1<p_2<q_2<\infty$,  $1/q_1+1/p_1'=\lambda_1/n$ and $1/q_2+1/p_2'=\lambda_2/m$,
the weights $\mu_1,\sigma_1 \in A_{p_1,q_1}(\mathbb R^n)$, $\mu_2,\sigma_2 \in A_{p_2,q_2}(\mathbb R^m)$ and $\nu:=\mu_1\sigma_1^{-1}\otimes \mu_2\sigma_2^{-1}$,  $I_{\lambda_1}^1$  stands for $I_{\lambda_1}$ acting on the first variable and $I_{\lambda_2}^2$ stands for $I_{\lambda_2}$ acting on the second variable, $\BMO_{\rm{prod}}(\nu)$ is a weighted product $\BMO$ space and $L^{p_2}(L^{p_1})(\mu_2^{p_2}\times\mu_1^{p_1})$ and $ L^{q_2}(L^{q_1})(\sigma_2^{q_2}\times\sigma_1^{q_1}) $ are mixed-norm spaces.
\end{abstract}



\section{Introduction and The Main Results}
Let $\mu$ and $\sigma$ be two weights in $\mathbb R^{n+m}$. A two-weight problem asks for a characterization of the boundedness of an operator $T: L^{p}(\mu)\rightarrow L^p(\sigma)$.
In a Bloom type variant of this problem, $\mu$ and $\sigma$ are Muckenhoupt $A_p$ weights and  a function $b$, which   is taken from some appropriate weighted $\BMO$ space $\BMO(\nu)$ for some Bloom type weight $\nu:= \mu^{1/p}\sigma^{-1/p}$,
is invoked.
This leads us naturally to the commutator setting.

In the one-parameter case, Bloom \cite{Bloom} obtained such a two-weight estimate for $[b, H]$, where $H$ is the Hilbert transform. Holmes, Lacey and Wick \cite{HLW} extended Bloom's result to general Calder\'on-Zygmund operators. The iterated case is by Holmes and Wick \cite{HW} (see also Hyt\"onen \cite{Hytonen16}). An improved iterated case is by Lerner, Ombrosi and Rivera-R\'ios \cite{LORR}.

In the bi-parameter case, there are two types of commutators: one is involved with little BMO spaces, and the other is associated with the more complicated product BMO spaces.
For the first type, Holmes, Petermichl and Wick~\cite{HolmesPetermichlWick} initiated the study of $[b, T]$ with $T$ being any bi-parameter Calder\'on-Zygmund operator and $b$ being some weighted little BMO function. Then Li, Martikainen and Vuorinen~\cite{LiMartikainenVuorinen1806}
extended Holmes-Petermichl-Wick's result to higher order commutators, and provided a simpler proof for the first order case. In \cite{cg} Cao and Gu studied the related question for fractional integrals. For the second type, Li, Martikainen and Vuorinen~\cite{LiMartikainenVuorinen1810} first studied the question for
$[T_n, [b, T_m]]$ (known as the Ferguson-Lacey type commutator \cite{Ferguson-Lacey}), where $T_n$ and $T_m$ are one-parameter Calder\'on-Zygmund operators in $\mathbb R^n$ and $\mathbb R^m$, respectively. Recently, Airta~\cite{Airta} generalized this result to the multi-parameter case.

In this paper, we focus on the Bloom type inequality for the Ferguson-Lacey type commutator involved with fractional integrals.
Specifically, we prove a Bloom type inequality for $[I_{\lambda_1}^1,[b,I_{\lambda_2}^2]]$, where
\begin{equation}\label{eq:e1}
I_{\lambda}f(x) := \int \frac{f(y)}{|x-y|^{\lambda}}dy
\end{equation}
is the fractional integral operator,
and for a measurable function $f$ defined on $\mathbb R^{n+m}$,
$I_{\lambda_1}^1f$ stands for  $I_{\lambda_1}$ acting on the first variable of $f$
and $I_{\lambda_2}^2f$ stands for $I_{\lambda_2}$ acting on the second variable of $f$, i.e.,
\[
I_{\lambda_1}^1f(x_1,x_2):=I_{\lambda_1}\big(f(\cdot,x_2)\big)(x_1), \quad I_{\lambda_2}^2f(x_1,x_2):=I_{\lambda_2}\big(f(x_1,\cdot)\big)(x_2).
\]
Our results extend similar results for singular integral operators.
The main difference is that  mixed-norm spaces \cite{Benedek1962,BenedekPanzone}
are invoked when we study the off-diagonal case of Bloom type inequalities
for fractional integral operators.
Note that there are many results on mixed-norm spaces, e.g., see
\cite{ChenSun2017,
Fernandez1987,
Hart2018,
Hormander1960,Rubio1986,
Stefanov2004,Torres2015} for some recent advances on mixed-norm Lebesgue spaces,
 \cite{Cleanthous2017c} for Besov spaces,
 \cite{Georgiadis2017,Johnsen2015} for Triebel-Lizorkin spaces
and \cite{Cleanthous2017,HuangLiuYangYuan2018,Huang2019}
for Hardy spaces.

To get a Bloom type inequality for singular integral operators,
a basic tool is the
representation theorem (see Hyt\"{o}nen \cite{Hytonen10} for the one-parameter case, and Martikainen \cite{Martikainen2011}, Ou \cite{Ou2015} for the bi-parameter and multi-parameter cases, respectively).
To deal with fractional integral operators, we need to establish a representation formula.
Before stating our main results, we introduce some notations.

Let $\mathscr D^0$ be the standard dyadic system in $\mathbb R^n$, i.e.,
\[
\mathscr{D}^{0}:=\bigcup_{k\in\mathbb{Z}}\mathscr{D}^{0}_{k},\quad \mathscr{D}^{0}_{k}:=\Big\{2^{-k}\big([0,1)^{n}+m\big):m\in \mathbb{Z}^{n}\Big\}.
\]
Given some $\omega=(\omega_j)_{j\in\mathbb{Z}}\in(\{0,1\}^n)^{\mathbb{Z}}$ and a cube $I$, denote
\[
I\dotplus \omega := I+\sum_{j:2^{-j}<\ell(I)}2^{-j}\omega_j.
\]
We define the random dyadic system $\mathscr{D}^{\omega}$ by
\begin{equation}\label{eq:Dw}
\mathscr{D}^{\omega}:= \big\{I\dotplus w:I\in \mathscr{D}^{0}\big\}=\bigcup_{k\in\mathbb{Z}}\mathscr{D}^{\omega}_{k}.
\end{equation}

\begin{definition}[Fractional dyadic shifts]
Given two integers $i,j\ge 0$, the fractional dyadic shift associated with $(i,j)$ is defined by
\[
S^{i,j}_{\lambda,\omega}f := \sum_{K\in\mathscr{D}^{\omega}} A^{i,j}_{\lambda,K}f,
\]
where each $A^{i,j}_{\lambda,K}$ has the form
\[
A^{i,j}_{\lambda,K}f = \sum_{\substack{I,J\in\mathscr{D}^{\omega} \\ I^{(i)}=K, J^{(j)}=K}}a_{\lambda,I,J,K}\langle f,h_{I}\rangle h_{J}
\]
with the coefficients $a_{\lambda,I,J,K}$ satisfying
$|a_{\lambda,I,J,K}| \le   |I|^{1/2}|J|^{1/2}/|K|^{\lambda/n}$.
\end{definition}

We are now ready to state our first main result, a representation formula
of fractional integral operators.
\begin{theorem}\label{thm:rep}
Let $0<\lambda<n$ be a constant and  $I_{\lambda}$ be a fractional integral operator
defined by (\ref{eq:e1}). Then we have
 \[
 \langle g, I_{\lambda}f\rangle = C\cdot \mathbb{E}_{\omega}\sum^{\infty}_{i.j=0} 2^{-\max{(i,j)}/2}\langle g, S^{i,j}_{\lambda,\omega}f\rangle,
 \]
where
 $f,g \in L^{\infty}_{c}(\mathbb{R}^n)$ and
 the constant $C$ depends only on $\lambda$ and the dimension $n$.
\end{theorem}

Recall that for $1<p<\infty$,  the Muckenhoupt $A_p$ class consists of all
locally integrable  positive functions $w(x)$ for which
\[
[w]_{A_p(\mathbb R^n)}:=\sup_{Q}\bve{w}_{Q}\bve{w^{1-p'}}_{Q}^{p-1}<\infty.
\]
And for $1<p<q<\infty$,   $A_{p,q}(\mathbb R^n)$ consists of all weight functions
$w(x)$ for which
\[
[w]_{A_{p,q}(\mathbb R^n)}:=\sup_{Q}\bve{w^q}_{Q}\bve{w^{-p'}}_{Q}^{q/p'}<\infty,
\]
where $\ave{w}_{Q}:=\frac{1}{Q}\int_{Q}w$ and the supremum is taken over all cubes $Q\subset\mathbb R^n$ with sides parallel to the axes.

Now we introduce  the weighted product BMO space. Let $\Dw$ be a dyadic system in $\mathbb R^n$, where $\omega=(\omega_j)_{j\in\mathbb{Z}}\in(\{0,1\}^n)^{\mathbb{Z}}$.
And let $\mathscr D^{\beta}$ be a dyadic system in $\mathbb R^m$,
where
$\beta=(\beta_j)_{j\in\mathbb{Z}}\in (\{0,1\}^m)^{\mathbb{Z}}$.
Given $w \in A_2(\mathbb R^{n+m})$, we say that a locally integrable
function $b:\mathbb R^{n+m}\rightarrow\mathbb C$ belongs to the weighted
product BMO space $\BMO_{\pro}^{\Dw,\Db}(w)$ if
\[
\|b\|_{\BMO_{\pro}^{\Dw,\Db}(w)}:=\sup_{\Omega}\Big(\frac{1}{w(\Omega)}\sum_{\substack{I\in\Dw,J\in\Db\\I\times J\subset \Omega}}|\pair{b}{h_I\otimes h_J}|^2\ave{w}_{I\times J}^{-1}\Big)^{1/2}<\infty,
\]
where $w(\Omega):=\int_{\Omega}w(x_1,x_2)dx_1dx_2$ and
the supremum is taken over all subsets $\Omega\subset \mathbb R^{n+m}$
such that $|\Omega|<\infty$ and   for every $x\in \Omega$,
there exist some $I\in\Dw$ and $J\in\Db$ satisfying $x\in I\times J\subset\Omega$.
The non-dyadic product BMO norm
can be defined  by taking the supremum over all dyadic systems $\Dw$ and $\Db$, i.e.,
\[
 \|b\|_{\BMO_{\pro}(w)} := \sup_{\Dw,\Db} \|b\|_{\BMO_{\pro}^{\Dw,\Db}(w)}.
\]

Now we   state the second main result, a Bloom type inequality for fractional integral operators.
\begin{theorem}\label{thm:blo}
Let $I_{\lambda_1}$ and $I_{\lambda_2}$ be two fractional integral operators
acting on functions defined on $\mathbb R^n$ and $\mathbb R^m$, respectively.
Suppose that $1<p_1<q_1<\infty$, $1<p_2<q_2< \infty$,
$1/q_1+1/p_1'=\lambda_1/n$
and $1/q_2+1/p_2'=\lambda_2/m$.
Let $\mu_1,\sigma_1\in A_{p_1,q_1}(\mathbb R^n)$ and $\mu_2,\sigma_2
\in A_{p_2,q_2}(\mathbb R^m)$. Set $\nu := \mu_1\sigma_1^{-1}
\otimes \mu_2\sigma_2^{-1}$. Then we have the quantitative estimate
\begin{align*}
\mbox{\hbox to 8em{}}&\hskip -8em \big\|\big[I_{\lambda_1}^1,\big[b,I_{\lambda_2}^2\big]\big]\big\|_{L^{p_2}(L^{p_1})
 (\mu_2^{p_2}\times\mu_1^{p_1})\rightarrow L^{q_2}(L^{q_1})(\sigma_2^{q_2}\times\sigma_1^{q_1})} \\
&
\lesssim_{\substack{[\mu_1]_{A_{p_1,q_1}(\mathbb R^n)},[\mu_2]_{A_{p_2,q_2}(\mathbb R^m)} \\ [\sigma_1]_{A_{p_1,q_1}(\mathbb R^n)},[\sigma_2]_{A_{p_2,q_2}(\mathbb R^m)}}}\|b\|_{\BMO_{\pro}(\nu)}.
\end{align*}
\end{theorem}

Here the weight is of tensor product type. We do not know how to relax this restriction.
This is mainly due to the natural appearance of the mixed-norm spaces,
in which even the boundedness of the strong maximal function with non-tensor product type weights is still open.

The paper is organized as follows.
In Section 2, we collect some preliminary results.
And in Section 3, we give a proof of Theorem~\ref{thm:rep}.  In Section 4, we
present  some results on  mixed-norm spaces and then give a proof of
the Bloom type inequality for fractional integral operators.

\section{Notations and Preliminary Results}

We denote $A\lesssim B$ if $A\leq C\cdot B$ for some constant $C$ that can depend on the dimension of the underlying spaces, on integration exponents, and on various other constants appearing in the assumptions.

We denote the product space $\mathbb R^{n+m}=\mathbb R^n\times\mathbb R^m$.
For any $x\in\mathbb R^{n+m}$, we write $x=(x_1,x_2)$ with $x_1\in\mathbb R^n$ and $x_2\in\mathbb R^m$.

Let $\mu_1$ and $\mu_2$ be measures on $\mathbb R^n$ and $\mathbb R^m$, respectively. For $p_1,p_2$ with $1\leq p_1,p_2<\infty$ and a measurable function $f:\mathbb R^{n+m}\rightarrow \mathbb C$, we define the mixed-norm $\|f\|_{L^{p_2}(L^{p_1})(\mu_2\times\mu_1)}$ by
\begin{align*}
\|f\|_{L^{p_2}(L^{p_1})(\mu_2\times\mu_1)}:=&\bigg(\int_{\mathbb R^m}\bigg(\int_{\mathbb R^n}|f(x_1,x_2)|^{p_1}d\mu_1(x_1)\bigg)^{p_2/p_1}d\mu_2(x_2)\bigg)^{1/p_2}\\
& = \big\|\|f(\cdot,x_2)\|_{L^{p_1}(\mu_1)}\big\|_{L^{p_2}(\mu_2)}.
\end{align*}

\subsection{Random dyadic systems}

Let $\mathscr{D}^{\omega}$ be defined by (\ref{eq:Dw}).
For parameters $r \in \mathbb{Z}_{+}$ and $\gamma\in(0,1/2)$, we call a cube $I\in\mathscr{D}^{\omega}$   bad if there exists some $J\in \mathscr{D}^{\omega}$ such that  $\ell(J)\geq 2^{r}\ell(I)$ and
\[
\dist(I,\partial J) \leq \ell(J)\cdot \Big(\frac{\ell(I)}{\ell(J)}\Big)^{\gamma},
\]
where $\ell(I)$ stands for the side length. In the treatment of a fractional integral operator $I_{\lambda}$ with exponent $\lambda$, the choice $\gamma =\frac{1}{2(\lambda+1)}$ is useful.

We say that a cube  $I\in\mathscr{D}^{\omega}$ is good
if it is not bad. It is well-known that $\mathbb P(\{\omega: I\dotplus \omega \ \ \text{is good}\})$ is independent of the choice of $I\in \mathscr D^0$ and if $r$ is sufficiently large, then
\[
\mathbb P_{\operatorname{good}}:= \mathbb P(\{\omega: I\dotplus \omega \ \ \text{is good}\})>0.
\]
Given two cubes $I,J\in\mathscr{D}^{\omega}$, we denote by $I\vee J$ the smallest cube in $\mathscr{D}^{w}$ that contains both $I$ and $J$. If it does not exist, we denote $I\vee J=\emptyset$.

\subsection{Haar functions and martingale differences}
In one dimension, for an interval $I$, the Haar functions are defined by $h^{0}_{I} = |I|^{-1/2}1_{I}$ and $h^{1}_{I} = |I|^{-1/2}(1_{I_{\ell}}-1_{I_{r}})$, here $I_{\ell}$ and $I_{r}$ are the left and right halves of the interval $I$, respectively. In higher dimensions,
for a cube $I=I_1\times\cdots \times I_n\subset \mathbb{R}^n$,  we  define the Haar functions $h^{\eta}_{I}$ as
\begin{equation}\label{eq:Haar0}
h^{\eta}_{I}(x) = h^{(\eta_{1},\cdots , \eta_{n})}_{I_{1}\times\cdots\times I_{n}}(x_{1},\cdots,x_{n}):=\otimes^{n}_{i=1} h^{\eta_{i}}_{I_{i}}(x_{i}),
\end{equation}
where $\eta\in \{0,1\}^{n}\setminus\{0\}$.

For a locally integrable function $f:\mathbb R^{n}\rightarrow \mathbb C$, the martingale difference $\Delta_I$ associated with a dyadic cube $I\in\Dw$ is defined by
\[
\Delta_I f:=
\sum_{\eta\in \{0,1\}^{n}\setminus\{0\}}\pair{f}{h_{I}^{\eta}}h_{I}^{\eta}.
\]
Set $\bve{f}_{I}:=\frac{1}{|I|}\int_{I}f$. We also write $E_{I}f:=\bve{f}_{I}1_{I}$. We have the usual martingale decomposition
\[
f =\sum_{I\in \mathscr D^\omega} \Delta_I f= \sum_{I\in \mathscr{D}^{\omega}}\sum_{\eta\in \{0,1\}^{n}\setminus\{0\}}\langle f,h^{\eta}_{I}\rangle h^{\eta}_{I}.
\]
Since the $\eta$'s do not play any major role, in the sequel  we just simply write
\begin{equation}\label{eq:haar}
f = \sum_{I\in \mathscr{D}^{\omega}}\langle f,h_{I}\rangle h_{I}.
\end{equation}

A martingale block is denoted by
\[
\Delta_{K,i}f = \sum_{\substack{I\in \Dw\\ I^{(i)}=K}}\Delta_I f, \quad \quad K\in\Dw, i\ge 0,
\]
where $I^{(i)}$ denotes the unique dyadic cube $P\in \Dw$ such that $I\subset P$ and $\ell(P)=2^i\ell(I)$.

By the definition of fractional dyadic shifts, it is  is easy to see that
\[
\big|S^{i,j}_{\lambda,\omega}f(x)\big| \leq \sum_{K\in\Dw} \frac{1}{|K|^{\lambda/n}}\int_{K}|f(y)|dy\cdot 1_{K}(x).
\]
 Cruz-Uribe and  Moen~\cite{UribeMoen} proved that
\begin{equation}\label{eq:frac int and max}
\sum_{K\in\Dw} \frac{1}{|K|^{\lambda/n}}\int_{K} |f(y)| dy\cdot 1_{K}(x) \lesssim I_{\lambda}|f|(x).
\end{equation}
Thus we have
\begin{equation}\label{eq:major}
\big|S^{i,j}_{\lambda,\omega}f(x)\big|\lesssim I_{\lambda}|f|(x).
\end{equation}

\subsection{Weights}
For any $1<p<q<\infty$  and $w \in A_{p,q}(\mathbb R^n) $, a simple calculation gives that $ w^q\in A_{q}(\mathbb R^n)$, $w^{-p'}\in A_{p'}(\mathbb R^n)$ and
$ w^{-q'}\in A_{q'}(\mathbb R^n)$.

In \cite{MuckenhouptWheeden1974}, Muckenhoupt and Wheeden proved the weighted bound for fractional integral operators. Specifically, they showed that if $0<\lambda<n$ and $1/q+1/p'=\lambda/n$, then
\[
\|I_{\lambda}f\|_{L^q(w^q)}\lesssim_{[w]_{A_{p,q}(\mathbb R^n)}}\|f\|_{L^p(w^p)}
\]
if and only if $w(x)\in A_{p,q}(\mathbb R^n)$.
Combining with the inequality (\ref{eq:major}), we have
\begin{equation}\label{eq:s:e14}
\|S^{i,j}_{\lambda,\omega}f\|_{L^q(w^q)}\lesssim_{[w]_{A_{p,q}(\mathbb R^n)}}\|f\|_{L^p(w^p)}.
\end{equation}

\subsection{Maximal functions and martingale difference square functions}
Let $\Dw$ be a dyadic system in $\mathbb R^n$, where $\omega=(\omega_j)_{j\in\mathbb{Z}}\in(\{0,1\}^n)^{\mathbb{Z}}$.
For a measurable function defined on $\mathbb R^n$,
we define
the martingale difference square function $S_{\Dw}f$ by
\[
S_{\Dw}f:=\Big(\sum_{I\in\Dw}|\Delta_I f|^2\Big)^{1/2}.
\]

Wilson~\cite{Wilson} proved the
following weighted estimates for martingale difference square functions.

\begin{proposition}[{\cite[Theorem 2.1]{Wilson}}]\label{prop:stand square}
For any $1<p<\infty$, $w\in A_p(\mathbb R^n)$ and
a measurable function  $f$,  we have the following norm equivalence,
\[
\|f\|_{L^p(w)}\thicksim_{[w]_{A_p(\mathbb R^n)}}\|S_{\Dw}\|_{L^p(w)}.
\]
\end{proposition}

And Cruz-Uribe, Martell and  P\'{e}rez~\cite{UribeMartellPerez2012} gave the following sharp  estimate.

\begin{proposition}[{\cite[Theorem 1.8]{UribeMartellPerez2012}}] \label{prop:stand mds func}
Given $p, 1<p<\infty$, then for any measurable function $f$ and
$w\in A_{p}(\mathbb R^n)$,
\[
\|S_{\Dw}f\|_{L^p(w)}\leq C_{n,p}[w]_{A_{p}(\mathbb R^n)}^{\max(\frac{1}{2},\frac{1}{p-1})}\|f\|_{L^p(w)}.
\]
Further, the exponent $\max(\frac{1}{2},\frac{1}{p-1})$ is the best possible.
\end{proposition}

Next we introduce some   notations on the product space $\mathbb R^{n+m}$.
Given a measurable function $f$ defined on $\mathbb R^{n+m}$, we define the strong maximal function $M_{S}$ by
\[
M_{S}f(x_1,x_2):=\sup_{R\ni (x_1,x_2)}\frac{1}{|R|}\int_{R}|f(y_1,y_2)|dy_1dy_2,
\]
where $R=Q\times Q'$ and $Q\subset\mathbb R^n$ and $Q'\subset\mathbb R^m$ are cubes with sides parallel to the axes.

Let $\Dw$ and  $\mathscr D^{\beta}$
be dyadic systems in $\mathbb R^n$ and $\mathbb R^m$, respectively, where $\omega=(\omega_j)_{j\in\mathbb{Z}}\in(\{0,1\}^n)^{\mathbb{Z}}$ and   $\beta=(\beta_j)_{j\in\mathbb{Z}}\in (\{0,1\}^m)^{\mathbb{Z}}$.
 We define the dyadic maximal functions by
\[
M_{\Dw}^1f(x_1,x_2):=\sup_{I\ni x_1,I\in\Dw} \frac{1}{|I|}\int_{I}|f(y_1,x_2)|dy_1,
\]
\[
M_{\Db}^2f(x_1,x_2):=\sup_{J\ni x_2,J\in\Db} \frac{1}{|J|}\int_{J}|f(x_1,y_2)|dy_2,
\]
\[
M_{\Dw,\Db}f(x_1,x_2):=\sup_{\substack{I\ni x_1,I\in\Dw \\  J\ni x_2, J\in\Db}}\frac{1}{|I|\times|J|}\int_{J}\int_{I}|f(y_1,y_2)|dy_1dy_2.
\]
It is obvious that $M^1_{\Dw}f \leq M_{S}f$,
$M^2_{\Db}f \leq M_{S}f$ and $M_{\Dw,\Db}f \leq M_{S}f$, thanks to
the Lebesgue differentiation theorem.

Let $0<\lambda_1<n$ and $0<\lambda_2<m$. We define the partial
fractional maximal functions by
\[
M_{\lambda_1,\Dw}^1f(x_1,x_2):=\sup_{I\ni x_1,I\in\Dw} \frac{1}{|I|^{\lambda_1/n}}\int_{I}|f(y_1,x_2)|dy_1.
\]
\[
M_{\lambda_2,\Db}^2f(x_1,x_2):=\sup_{J\ni x_2,J\in\Db} \frac{1}{|J|^{\lambda_2/m}}\int_{J}|f(x_1,y_2)|dy_2.
\]
By \eqref{eq:frac int and max}, it is easy to see that
\begin{equation}\label{eq:s:e5}
  M_{\lambda_1,\Dw}^1f(x_1,x_2)\lesssim I_{\lambda_1}^1|f|(x_1,x_2).
\end{equation}
Similarly,
\begin{equation}\label{eq:s:e6}
  M_{\lambda_2,\Db}^2f(x_1,x_2)\lesssim I_{\lambda_2}^2|f|(x_1,x_2).
\end{equation}

And we define the martingale difference square functions on the product space
by
\begin{align*}
S_{\Dw}^1f&=\Big(\sum_{I\in\Dw}|\Delta_{I}^1f|^2\Big)^{1/2},\\ S_{\Db}^2f&=\Big(\sum_{J\in\Db}|\Delta_{J}^2f|^2\Big)^{1/2},\\   S_{\Dw,\Db}f&=\Big(\sum_{\substack{I\in\Dw\\J\in\Db}}|\Delta_{I\times J}f|^2\Big)^{1/2},
\end{align*}
where $
\Delta_{I}^1f(x) =\Delta_{I}(f(\cdot,x_2))(x_1)$,
$\Delta_{J}^2f(x) =\Delta_{J}(f(x_1,\cdot))(x_2)$ and $\Delta_{I\times J}f(x) =\Delta_J^{2}(\Delta_I^1f)(x)$.

Similarly we set $E_{I}^1:=E_{I}(f(\cdot,x_2))(x_1)$ and $E_{J}^2:=E_{J}(f(x_1,\cdot))(x_2)$.

Notice that $\Delta_I^1f=h_I\otimes \pair{f}{h_I}_1,\Delta_J^2=\pair{f}{h_J}_2\otimes h_J$ and $\Delta_{I\times J}f=\pair{f}{h_I\otimes h_J}h_I\otimes h_J$, where $\pair{f}{h_I}_1$ is defined by
\[
\pair{f}{h_I}_1(x_2)=\int_{\mathbb R^n}f(y_1,x_2)h_I(y_1)dy_1,
\]
and $\pair{f}{h_J}_2$ is defined similarly.

The Martingale blocks are defined in the natural way,
\[
\Delta_{K\times V}^{i,j}f=\sum_{\substack{I\in\Dw\\I^{(i)}=K}}\sum_{\substack{J\in\Db\\J^{(j)}=V}}\Delta_{I\times J}f.
\]

\section{Proof of Theorem~\ref{thm:rep}}

In this section, we give a proof of Theorem~\ref{thm:rep}.
First, we introduce a result by Hyt\"onen~\cite{Hytonen17}.

\begin{lemma}[{\cite[Lemma 3.7]{Hytonen17}}]\label{lem:majorant}
Let $I,J\in\mathscr{D}^{\omega}$ be such that $I$ is good, $I\cap J=\emptyset$ and $\ell(I)\leq \ell(J)$. Then there exists some $K\supset I\cup J$ which satisfies
\begin{align*}
&\ell(K) \leq 2^r\ell(I),  &\text{if}\   \dist(I,J)\leq \ell(J)\cdot \Big(\frac{\ell(I)}{\ell(J)}\Big)^{1/2(\lambda +1)},\\
&\ell(K)\cdot \Big(\frac{\ell(I)}{\ell(K)}\Big)^{1/2(\lambda +1)} \leq 2^r \dist(I,J),  &\text{if} \   \dist(I,J)> \ell(J)\cdot \Big(\frac{\ell(I)}{\ell(J)}\Big)^{1/2(\lambda +1)}.
\end{align*}
\end{lemma}

To prove the main result, we also need the following lemma.

\begin{lemma}\label{lem:vanish}
Let  $I$ and $J$ be cubes in $\mathbb{R}^n$. If $I$ and $J$ share the same center, then
 $
 \langle 1_{J},I_{\lambda}h_{I} \rangle = 0
 $.
\end{lemma}

\begin{proof}
For $\eta\in \{0,1\}^{n}\setminus\{0\}$, let the  Haar functions
$h^{\eta}_{I}$ be defined by (\ref{eq:Haar0}).
Observe that there exists some $i$  such that $\eta_{i}=1$.
Without loss of generality, we  assume that  $i=1$.
We denote the center of the interval $I_1$ by $c$.
Moreover, we further divide cubes $I$ and $J$ into the following parts,
\[
J^1 = \{x:x\in J,x_1>c\},\qquad J^2 = \{x:x\in J,x_1<c\},
\]
\[
I^1 = \{y:y\in I,y_1>c\},\qquad I^2 = \{y:y\in I,y_1<c\}.
\]
Then
\begin{align*}
\langle 1_{J},I_{\lambda}h_{I} \rangle
&= \langle 1_{J^1},I_{\lambda}(1_{I^1}\cdot h_{I}) \rangle+\langle 1_{J^2},I_{\lambda}(1_{I^2}\cdot h_{I}) \rangle \\
&\qquad +\langle 1_{J^1},I_{\lambda}(1_{I^2}\cdot h_{I}) \rangle+\langle 1_{J^2},I_{\lambda}(1_{I^1}\cdot h_{I}) \rangle\\
&=: A_1+A_2+A_3+A_4,
\end{align*}
where
\[
A_1 = \int_{J^1}\int_{I^1}\frac{h_{I}(y)}{|x-y|^{\lambda}}dydx
\]
and
\begin{align*}
A_2 =& \int_{J^2}\int_{I^2}\frac{h_{I}(y)}{|x-y|^{\lambda}}dydx\\
=& \int_{J^1}\int_{I^1}\frac{h_{I}(2c-y_1,y_2,y_3,\cdots, y_n)}{|(2c-x_1-(2c-y_1),x_2-y_2,x_3-y_3,\cdots,x_n-y_n)|^{\lambda}}dydx\\
=& \int_{J^1}\int_{I^1}\frac{-h_{I}(y)}{|x-y|^{\lambda}}dydx\\
=& -A_1.
\end{align*}
So we have $A_1+A_2=0$. Similarly we can show that $A_3+A_4=0$. Therefore,
\[
\langle 1_{J},I_{\lambda}h_{I} \rangle = 0.
\]
This completes the proof.
\end{proof}

We are now ready to give a proof of Theorem~\ref{thm:rep}.

\begin{proof}[Proof of Theorem \ref{thm:rep}]
Using \eqref{eq:haar} to expand $f$ and $g$, we have
\begin{align*}
\langle g, I_{\lambda}f\rangle&=  \mathbb{E}_{\omega}\sum_{ I,J\in\mathscr{D}^{\omega} }\langle g,h_{J} \rangle \langle h_{J},I_{\lambda}h_{I}\rangle \langle h_{I},f\rangle\\
&=  \mathbb{E}_{\omega}\sum_{\substack{I,J\in\mathscr{D}^{\omega}\\ \ell(I)\leq \ell(J) }}\langle g,h_{J} \rangle \langle h_{J},I_{\lambda}h_{I}\rangle \langle h_{I},f\rangle+ \mathbb{E}_{\omega}\sum_{\substack{I,J\in\mathscr{D}^{\omega}\\ \ell(I)> \ell(J) }}\langle g,h_{J} \rangle \langle h_{J},I_{\lambda}h_{I}\rangle \langle h_{I},f\rangle.
\end{align*}
Since the goodness of a cube is independent of its position, we have
\begin{align*}
\langle g, I_{\lambda}f\rangle&= \pi_{\operatorname{good}}^{-1} \mathbb{E}_{\omega}\sum_{\substack{I,J\in\mathscr{D}^{\omega}\\ \ell(I)\leq \ell(J) \\ I \ \text{is good}}}\langle g,h_{J} \rangle \langle h_{J},I_{\lambda}h_{I}\rangle \langle h_{I},f\rangle  \\
&\qquad +  \pi_{\operatorname{good}}^{-1}\mathbb{E}_{\omega}\sum_{\substack{I,J\in\mathscr{D}^{\omega}\\ \ell(I)> \ell(J)\\ J \ \text{is good} }}\langle g,h_{J} \rangle \langle h_{J},I_{\lambda}h_{I}\rangle \langle h_{I},f\rangle\\
&=: \text{I}+\text{II}.
\end{align*}

First, we estimate the first term  $\text I$.
We have
\begin{align*}
&\sum_{\substack{I,J\in\mathscr{D}^{\omega}\\ \ell(I)\leq \ell(J)\\ I\ \text{is good}}}
 \langle g,h_{J} \rangle \langle h_{J},I_{\lambda}h_{I}\rangle \langle h_{I},f\rangle \\
 &=
 \sum_{\substack{I,J\in\mathscr{D}^{\omega}, \ \ell(I)\leq \ell(J)\\ \dist(I,J) > \ell(J)\cdot (\ell(I)/\ell(J))^{1/2(\lambda+1)}\\ I\  \text{is good}}}
  + \sum_{\substack{I,J\in\mathscr{D}^{\omega}, \ \ell(I)\leq \ell(J)\\ \dist(I,J) \leq \ell(J)\cdot (\ell(I)/\ell(J))^{1/2(\lambda+1)}\\ I\cap J = \varnothing \\ I \ \text{is good}}}\\
 &\qquad +\sum_{\substack{I,J\in\mathscr{D}^{\omega}, \ \ell(I)\leq \ell(J)\\ I\subseteq J;\ell(I)\geq 2^{-r}\ell(J)\\ I\ \text{is good}}}
+ \sum_{\substack{I,J\in\mathscr{D}^{\omega}, \ \ell(I)\leq \ell(J)\\
 I\varsubsetneq J;\ell(I)< 2^{-r}\ell(J)\\ I\ \text{is good}}}\\
 &=:\sigmaout+\sigmanear+\sigmashallow+\sigmadeep.
\end{align*}
Now we estimate the four terms separately.

{\textbf{The term $\sigmaout$}}. By Lemma~\ref{lem:majorant}, we know that
$K=I\vee J$ exists and
\[
\ell(K)\Big( \frac{\ell(I)}{\ell(K)}\Big)^{1/2(\lambda+1)}\leq 2^r \dist(I,J).
\]
We write
\[
\sigmaout = \sum^{\infty}_{j=1} \sum^{\infty}_{i=j} \sum_{K\in \mathscr{D}^{w}} \sum_{\substack{I,J\in\mathscr{D}^{w}\\I\vee J =K;I^{(i)}=K,J^{(j)}=K\\ \dist(I,J) > \ell(J)\cdot (\ell(I)/\ell(J))^{1/2(\lambda+1)}\\ I\ \text{is good}}}\langle g,h_{J} \rangle \langle h_{J},I_{\lambda}h_{I}\rangle \langle h_{I},f\rangle.
\]
Denote the center of the cube $I$ by $y_{I}$. We have
\begin{align*}
|\langle h_{J},I_{\lambda}h_{I}\rangle|  =& \Big|\int_{J}\int_{I}\frac{h_{J}(x)h_{I}(y)}{|x-y|^{\lambda}}dydx\Big|\\
=&\Big|\int_{J}\int_{I}\big(\frac{1}{|x-y|^{\lambda}}-\frac{1}{|x-y_{I}|^{\lambda}}\big)h_{J}(x)h_{I}(y)dydx\Big|\\
\lesssim & \|h_{J}\|_{1} \|h_{I}\|_{1}\frac{\ell(I)}{\dist(I,J)^{\lambda+1}}\\
\lesssim & |I|^{1/2}|J|^{1/2}\frac{\ell(I)}{\ell(K)^{\lambda+1}(\ell(I)/\ell(K))^{1/2}}\\
=& \frac{|I|^{1/2}|J|^{1/2}}{|K|^{\lambda/n}}\cdot \frac{\ell(I)^{1/2}}{\ell(K)^{1/2}}\\
=& 2^{-i/2}\cdot \frac{|I|^{1/2}|J|^{1/2}}{|K|^{\lambda/n}}.
\end{align*}
So we obtain
\begin{align*}
\sigmaout =& c\cdot \sum^{\infty}_{j=1}\sum^{\infty}_{i=j}2^{-i/2}\sum_{K\in\mathscr{D}^{w}}\langle g,A^{i,j}_{\lambda,K}f\rangle\\
=& c\cdot \sum^{\infty}_{j=1}\sum^{\infty}_{i=j}2^{-i/2}\langle g,S^{i,j}_{\lambda,w}f\rangle.
\end{align*}

{\textbf{The term $\sigmanear$}}.
Again by Lemma \ref{lem:majorant}, with the condition $I\cap J=\emptyset$ and $\dist(I,J)\leq \ell(J)(\ell(I)/\ell(J))^{1/(2\lambda+1)}$, we know that $K=I\vee J$ exists and  $\ell(K)\leq 2^r \ell(I)$. So $\sigmanear$ can be written as
\[
\sigmanear=\sum^{r}_{i=1} \sum^{i}_{j=i} \sum_{K\in \mathscr{D}^{w}} \sum_{\substack{I,J\in\mathscr{D}^{w},I\cap J=\emptyset\\I\vee J =K;I^{(i)}=K,J^{(j)}=K\\ \dist(I,J) \leq \ell(J)\cdot (\ell(I)/\ell(J))^{1/2(\lambda+1)}\\ I\ \text{is\ good}}}\langle g,h_{J} \rangle \langle h_{J},I_{\lambda}h_{I}\rangle \langle h_{I},f\rangle.
\]
Since $I_{\lambda}$ is bounded from $L^p$ to $L^q$, we have
\begin{align*}
|\langle h_{J},I_{\lambda}h_{I}\rangle|\leq & \|h_{J}\|_{q'}\|I_{\lambda}h_{I}\|_{q}\\
\lesssim & |J|^{1/q'}|J|^{-1/2}\|h_{I}\|_{p}\\
= & |J|^{1/q'}|J|^{-1/2}|I|^{1/p}|I|^{-1/2}.
\end{align*}
Observe that  $2^{-r}\ell(K)\leq \ell(I) \leq \ell(J) \leq \ell(K)$
and $1/{p'}+1/{q}= {\lambda}/{n}$. We have
\[
|J|^{1/q'}|J|^{-1/2}|I|^{1/p}|I|^{-1/2} \lesssim \frac{|I|^{1/2}|J|^{1/2}}{|K|^{\lambda/n}}.
\]
So
\begin{align*}
\sigmanear = & c\cdot\sum^{r}_{j=1}\sum^{r}_{i=j}\sum_{K\in\mathscr{D}^w}\langle g,A^{i,j}_{\lambda,K} \rangle
=  c\cdot\sum^{r}_{j=1}\sum^{r}_{i=j}\langle g, S^{i,j}_{\lambda,w} \rangle.
\end{align*}

{\textbf{The term $\sigmashallow$}}.
It is easy to see that the term $\sigmashallow$ can be written as
\[
\sigmashallow =\sum^{r}_{i=0}\sum_{\substack{J\in\mathscr{D}^w\\ I^{(i)}=J\\I\ \text{is\ good}}}\langle g,h_{J} \rangle \langle h_{J},I_{\lambda}h_{I}\rangle \langle h_{I},f\rangle.
\]
With similar arguments as those for calculating the term $\sigmanear$, we have
\begin{align*}
|\langle h_{J},I_{\lambda}h_{I}\rangle|\leq & \|h_{J}\|_{q'}\|I_{\lambda}h_{I}\|_{q}\\
\lesssim & |J|^{1/q'}|J|^{-1/2}\|h_{I}\|_{p}\\
= & |J|^{1/q'}|J|^{-1/2}|I|^{1/p}|I|^{-1/2}\\
\lesssim & \frac{|I|^{1/2}|J|^{1/2}}{|J|^{\lambda/n}}.
\end{align*}
Hence
\begin{align*}
\sigmashallow = & c\cdot\sum^{r}_{i=0}\sum_{J\in \mathscr{D}^w}\langle g ,A^{i,0}_{\lambda,J}f\rangle
=   c\cdot\sum^{r}_{i=0}\langle g, S^{i,0}_{\lambda,w}f\rangle.
\end{align*}

{\textbf{The term $\sigmadeep$}}.
This term   can be written as
\[
\sigmadeep = \sum^{\infty}_{i=r+1}\sum_{\substack{J\in\mathscr{D}^w\\ I^{(i)}=J\\I\ \text{is\ good}}}\langle g,h_{J} \rangle \langle h_{J},I_{\lambda}h_{I}\rangle \langle h_{I},f\rangle.
\]
The cube $J$ has $2^n$ children in $\mathscr{D}^{w}$. $I\subsetneq J$ implies that $I$ must be contained in one of the children $J_1$. Since $\ell(I)\leq2^{-r-1}\ell(J)$, we have $\ell(I)\leq2^{-r}\ell(J_1)$. Hence $I$ is good and
\begin{align*}
\dist(I,J_1) >& \ell(J_1)\big(\frac{\ell(I)}{\ell(J_1)}\big)^{1/2(\lambda+1)}\\
\gtrsim & \ell(J)\big(\frac{\ell(I)}{\ell(J)}\big)^{1/2(\lambda+1)}.
\end{align*}

We denote the center of the cube $I$ by $y_{I}$. Let $J'$ be the cube
with side length $2\ell(J)$ and the same center $y_I$. Then we have  $J'\supset J$. By Lemma~\ref{lem:vanish},
\[
\langle 1_{J'},I_{\lambda}h_{I} \rangle = 0.
\]
Now we see from the definition of Haar functions that  $h_{J}$ is a constant on $J_1$.
More precisely, $h_{J} \equiv \langle h_{J} \rangle_{J_1}$.
It follows that
\begin{align*}
\langle h_{J},I_{\lambda}h_{I} \rangle  = & \langle h_{J},I_{\lambda}h_{I} \rangle - \langle h_{J} \rangle_{J_1}\cdot \langle 1_{J'},I_{\lambda}h_{I} \rangle\\
= & \langle h_{J}-1_{J'}\cdot\langle h_{J} \rangle_{J_1} ,I_{\lambda}h_{I} \rangle\\
= & \langle 1_{J'\backslash J_1}\cdot (h_{J}-1_{J'}\cdot\langle h_{J} \rangle_{J_1}), I_{\lambda}h_{I}   \rangle.
\end{align*}
So we obtain
\begin{align*}
|\langle h_{J},I_{\lambda}h_{I} \rangle| = & \Big|\int_{J'\backslash J_1}\int_{I}\frac{(h_{J}-1_{J'}\cdot\langle h_{J} \rangle_{J_1})(x)h_{I}(y)}{|x-y|^{\lambda}}dydx\Big|\\
= & \Big| \int_{J'\backslash J_1}\int_{I}\big(\frac{1}{|x-y|^{\lambda}}-\frac{1}{|x-y_I|^{\lambda}}\big)(h_{J}-1_{J'}\cdot\langle h_{J} \rangle_{J_1})(x)h_{I}(y)dydx \Big|\\
\lesssim & \int_{J'\backslash J_1}\int_{I} \frac{\ell(I)}{\dist(I,J_1)^{\lambda+1}} |J|^{-1/2}|I|^{-1/2}dydx\\
\lesssim & |I|^{1/2}|J|^{1/2}\frac{\ell(I)}{\ell(J)^{\lambda+1}(\ell(I)/\ell(J))^{1/2}}\\
=& \frac{|I|^{1/2}|J|^{1/2}}{|J|^{\lambda/n}}\cdot\big( \frac{\ell(I)}{\ell(J)}\big)^{1/2}\\
=& 2^{-i/2} \cdot \frac{|I|^{1/2}|J|^{1/2}}{|J|^{\lambda/n}}.
\end{align*}
Hence
\begin{align*}
\sigmadeep =& c\cdot \sum^{\infty}_{i=r+1} 2^{-i/2} \sum_{J\in \mathscr{D}^w}\langle g, A^{i,0}_{\lambda,J}f \rangle
=  c\cdot \sum^{\infty}_{i=r+1} 2^{-i/2} \langle g,S^{i,0}_{\lambda,w}f \rangle.
\end{align*}

Similarly we can calculate the term $\text{II}$ and get the conclusion as desired.
This completes the proof.
\end{proof}

\section{Proof of Theorem~\ref{thm:blo}}
In this section, we give a proof of Theorem~\ref{thm:blo}.
First, we introduce some results on mixed-norm spaces.
We begin with a result on the weighted estimates of the strong maximal function.

\begin{proposition}[{\cite[Theorem 1]{Kurtz07}}]\label{pro:strong maximal}
Let $1<p,q<\infty$ and weights $w_1\in A_p(\mathbb R^n)$ and $w_2  \in A_q(\mathbb R^m)$. Then
\[
\|M_{S}f\|_{L^q(L^p)(w_2\times w_1)} \lesssim_{[w_1]_{A_p(\mathbb R^n)},[w_2]_{A_q(\mathbb R^m)}}\|f\|_{L^q(L^p)(w_2\times w_1)}.
\]
Moreover, we have the Fefferman-Stein inequality
\[
\Big\|\Big(\sum_{j} |M_{S}f_j|^2\Big)^{1/2}\Big\|_{L^q(L^p)(w_2\times w_1)} \!\!\lesssim_{[w_1]_{A_p(\mathbb R^n)},[w_2]_{A_q(\mathbb R^m)}}
\!\!\Big\|\Big(\sum_{j}|f_j|^2\Big)^{1/2}\Big\|_{L^q(L^p)(w_2\times w_1)}.
\]
\end{proposition}

It is easy to see that the above result remains true whenever  $M_S$
is replaced by $M_{\Dw}^1$, $M_{\Db}^2$ or $M_{\Dw,\Db}$.

Next we introduce the extrapolation theorem~\cite{Duoandikoetxea11}.

\begin{proposition} \label{prop:extropolation}
Assume that for a pair of nonnegative functions  $(f,g)$, for some $p_0\in [1,\infty)$ and for all $w\in A_{p_0}$, we have
\[
\Big(\int_{\mathbb R^n}g^{p_0}w\Big)^{1/p_0}\leq CN\big([w]_{A_{p_0}}\big)\Big(\int_{\mathbb R^n}f^{p_0}w\Big)^{1/p_0},
\]
where $N$ is an increasing function and the constant $C$ does not depend on $w$. Then for all $1<p<\infty$ and all $w\in A_p$, we have
\[
\Big(\int_{\mathbb R^n}g^{p}w\Big)^{1/p}\leq CK(w)\Big(\int_{\mathbb R^n}f^{p}w\Big)^{1/p},
\]
where $K(w)\leq C_1N\big(C_2[w]_{A_p}^{\max(1,\frac{p_0-1}{p-1})}\big)$ for $w\in A_p$.
\end{proposition}

With the above extrapolation theorem, we can prove the following estimates of
martingale difference square functions.

\begin{lemma}\label{Lm:square}
 Under the same hypotheses as in Proposition \ref{pro:strong maximal}, we have
\begin{align}
\|S_{\Dw}^1f\|_{L^q(L^p)(w_2\times w_1)} &\lesssim_{[w_1]_{A_p(\mathbb R^n)}}\|f\|_{L^q(L^p)(w_2\times w_1)},
   \label{eq:s:e1}
                \\
 \|S_{\Db}^2f\|_{L^q(L^p)(w_2\times w_1)} &\lesssim_{[w_2]_{A_q(\mathbb R^m)}}\|f\|_{L^q(L^p)(w_2\times w_1)},
   \label{eq:s:e2}
                \\
 \|S_{\Dw,\Db}f\|_{L^q(L^p)(w_2\times w_1)} &\lesssim_{[w_1]_{A_p(\mathbb R^n)},[w_2]_{A_q(\mathbb R^m)}}\|f\|_{L^q(L^p)(w_2\times w_1)}.
   \label{eq:s:e3}
\end{align}
\end{lemma}

\begin{proof} First  we   prove (\ref{eq:s:e1}). For any $f\in L^q(L^p)(w_2\times w_1)$,
we have
\begin{align*}
& \hskip -1em\|S_{\Dw}^1f\|_{L^q(L^p)(w_2\times w_1)} \\
&= \bigg(\int_{\mathbb R^m}\bigg(\int_{\mathbb R^n}|S_{\Dw}^1f(x_1,x_2)|^{p}w_1(x_1) dx_1 \bigg)^{q/p}w_2(x_2) dx_2\bigg)^{1/q}\\
&= \bigg(\int_{\mathbb R^m}\bigg(\int_{\mathbb R^n}\Big(\sum_{I\in\Dw}|\Delta_I f(\cdot,x_2)(x_1)|^2\Big)^{p/2}w_1(x_1)dx_1\bigg)^{q/p}w_2(x_2)dx_2\bigg)^{1/q}\\
&= \bigg(\int_{\mathbb R^m}\bigg(\int_{\mathbb R^n}S_{\Dw}\big(f(\cdot,x_2)\big)(x_1)^{p}w_1(x_1)dx_1\bigg)^{q/p}w_2(x_2)dx_2\bigg)^{1/q}.
\end{align*}
By Proposition~\ref{prop:stand square}, we have
\[
\bigg(\int_{\mathbb R^n}S_{\Dw}\big(f(\cdot,x_2)\big)(x_1)^{p}w_1(x_1)dx_2\bigg)^{1/p} \!\lesssim_{[w_1]_{A_p(\mathbb R^n)}}\! \bigg(\int_{\mathbb R^n}\! |f(x_1,x_2)|^{p}w_1(x_1) dx_1 \bigg)^{1/p}.
\]
Hence
\[
\|S_{\Dw}^1f\|_{L^q(L^p)(w_2\times w_1)} \lesssim_{[w_1]_{A_p(\mathbb R^n)}}\|f\|_{L^q(L^p)(w_2\times w_1)}.
\]

Next we   prove (\ref{eq:s:e2}).
Denote $\|f(\cdot,x_2)\|_{L^p(w_1)}$ and $\|S_{\Db}^2 f(\cdot,x_2)\|_{L^p(w_1)}$
by $F(x_2)$ and $G(x_2)$, respectively.
We have
\begin{align*}
 \bigg(\int_{\mathbb R^m}G^{q}w_2\bigg)^{1/q}&=\|S_{\Db}^2f\|_{L^q(L^p)(w_2\times w_1)},
   \\
  \bigg(\int_{\mathbb R^m}F^{q}w_2\bigg)^{1/q}&=\|f\|_{L^q(L^p)(w_2\times w_1)}.
\end{align*}
Notice that for any weight  $w\in A_p(\mathbb R^m)$, we have
\begin{align*}
\bigg(\int_{\mathbb R^m}G^{p} w \bigg)^{1/p}=&\bigg(\iint_{\mathbb R^{n+m}}\Big| S_{\Db}^2f(x_1,x_2)\Big|^p w_1(x_1) w(x_2) dx_1dx_2\bigg)^{1/p}\\
=& \bigg(\int_{\mathbb R^n}\int_{\mathbb R^m} \Big(\sum_{J\in\Db}|\Delta_J f(x_1,\cdot)(x_2)|^2\Big)^{p/2} w(x_2) w_1(x_1) dx_2dx_1\bigg)^{1/p}\\
=& \bigg(\int_{\mathbb R^n}\int_{\mathbb R^m} S_{\Db}\big(f(x_1,\cdot)\big)(x_2)^{p}w(x_2) w_1(x_1) dx_2dx_1\bigg)^{1/p}.
\end{align*}
By Proposition~\ref{prop:stand mds func},
\begin{align*}
& \bigg(\int_{\mathbb R^m} S_{\Db}\big(f(x_1,\cdot)\big)(x_2)^{p}w(x_2) dx_2\bigg)^{1/p}\\
&\le C_{m,p}[w]_{A_p(\mathbb R^m)}^{\max(\frac{1}{2},\frac{1}{p-1})}
\bigg(\int_{\mathbb R^m} |f(x_1,x_2)|^p w(x_2) dx_2 \bigg)^{1/p}.
\end{align*}
This give us
\begin{align*}
\bigg(\int_{\mathbb R^m}G^{p} w \bigg)^{1/p}
\leq & C_{m,p} [w]_{A_p(\mathbb R^m)}^{\max(\frac{1}{2},\frac{1}{p-1})} \bigg(\int_{\mathbb R^n}\int_{\mathbb R^m} |f(x_1,x_2)|^p w(x_2) w_1(x_1) dx_2 dx_1\bigg)^{1/p}\\
= & C_{m,p} [w]_{A_p(\mathbb R^m)}^{\max(\frac{1}{2},\frac{1}{p-1})} \bigg(\int_{\mathbb R^m} F^p w \bigg)^{1/p}.
\end{align*}
Note that $x^{\max(\frac{1}{2},\frac{1}{p-1})}$  is  increasing on $(0, \infty)$. By Proposition~\ref{prop:extropolation}, for any  $w\in A_q(\mathbb R^m)$, we have
\[
\bigg(\int_{\mathbb R^m} G^q w\bigg)^{1/q} \lesssim_{[w]_{A_q(\mathbb R^m)}} \bigg(\int_{\mathbb R^m} F^q w\bigg)^{1/q}.
\]
That is,
\[
\|S_{\Db}^2f\|_{L^q(L^p)(w_2\times w_1)} \lesssim_{[w_2]_{A_q(\mathbb R^m)}}\|f\|_{L^q(L^p)(w_2\times w_1)}.
\]

Finally we prove (\ref{eq:s:e3}).
By the Kahane-Khintchine inequality~\cite{HytonenNeervenVeraarWeis},
\begin{align*}
& \Big(\sum_{I\in\Dw,J\in\Db}|\Delta_{I\times J}f(x)|^2\Big)^{1/2}\\
=&\Big(\int_{\Omega}\sum_{J\in \Db}\Big|\sum_{I\in \Dw}\varepsilon_{I}(s)\Delta_{I\times J}f(x)\Big|^2 ds\Big)^{1/2}\\
\lesssim & \int_{\Omega} \Big(\sum_{J\in \Db}\Big|\Delta_{J}^2\Big(\sum_{I\in\Dw}\varepsilon_{I}(s)\Delta_{I}^1f(x)\Big)\Big|^2\Big)^{1/2}ds,
\end{align*}
where $(\varepsilon_I)_{I\in \Dw}$ is a Rademacher sequence defined on some  probability space $\big(\Omega,\mathbb P\big)$.

Since $p,q>1$, by  Minkowski's integral inequality,
\begin{align*}
& \bigg\|\int_{\Omega} \Big(\sum_{J\in \Db}\Big|\Delta_{J}^2\Big(\sum_{I\in\Dw}\varepsilon_{I}(s)\Delta_{I}^1f\Big)\Big|^2\Big)^{1/2}ds\bigg\|_{L^q(L^p)(w_2\times w_1)}\\
\leq & \int_{\Omega} \Big\| \Big(\sum_{J\in \Db}\Big|\Delta_{J}^2\Big(\sum_{I\in\Dw}\varepsilon_{I}(s)\Delta_{I}^1f\Big)\Big|^2\Big)^{1/2}\Big\|_{L^q(L^p)(w_2\times w_1)} ds\\
= & \int_{\Omega} \Big\| S_{\Db}^2 \Big(\sum_{I\in\Dw}\varepsilon_{I}(s)\Delta_{I}^1f\Big) \Big\|_{L^q(L^p)(w_2\times w_1)} ds.
\end{align*}
It follows from (\ref{eq:s:e2}) that
\[
\Big\| S_{\Db}^2 \Big(\sum_{I\in\Dw}\varepsilon_{I}(s)\Delta_{I}^1f\Big) \Big\|_{L^q(L^p)(w_2\times w_1)}\lesssim_{[w_2]_{A_q(\mathbb R^m)}}\Big\|\sum_{I\in\Dw}\varepsilon_{I}(s)\Delta_{I}^1f\Big\|_{L^q(L^p)(w_2\times w_1)}.
\]
On the other hand, we see from Proposition~\ref{prop:stand square} that
\begin{align*}
\Big\|\sum_{I\in\Dw}\varepsilon_{I}(s)\Delta_{I}^1f(\cdot,x_2)\Big\|_{L^p(w_1)} & \sim_{[w_1]_{A_p(\mathbb R^n)}} \Big\| S_{\Dw}\big( \sum_{I\in\Dw}\varepsilon_{I}(s)\Delta_{I}^1f(\cdot,x_2)\big)\Big\|_{L^p(w_1)}\\
& =  \Big\| S_{\Dw}\big(f(\cdot,x_2)\big)\Big\|_{L^p(w_1)}\\
& \sim_{[w_1]_{A_p(\mathbb R^n)}} \Big\| f(\cdot,x_2)\Big\|_{L^p(w_1)}.
\end{align*}
Combining the above inequalities, we get
\begin{align*}
  \|S_{\Dw,\Db}f\|_{L^q(L^p)(w_2\times w_1)}
& \lesssim  \int_{\Omega} \Big\| S_{\Db}^2 \Big(\sum_{I\in\Dw}\varepsilon_{I}(s)\Delta_{I}^1f\Big) \Big\|_{L^q(L^p)(w_2\times w_1)} ds\\
& \lesssim_{[w_2]_{A_q(\mathbb R^m)}}   \int_{\Omega} \Big\|\sum_{I\in\Dw}\varepsilon_{I}(s)\Delta_{I}^1f\Big\|_{L^q(L^p)(w_2\times w_1)} ds\\
& \lesssim_{[w_1]_{A_p(\mathbb R^n)},[w_2]_{A_q(\mathbb R^m)}} \|f\|_{L^q(L^p)(w_2\times w_1)}.
\end{align*}
This completes the proof.
\end{proof}

The   $\ell^2$-valued extension of linear operators on classical $L^p$ spaces is well known, e.g., see Grafakos~\cite[Chapter 4]{Grafakos}.
Here we give a similar result on mixed-norm spaces.
Since we  do not find a reference,   we present a short proof.

\begin{lemma}\label{Lm:vector value inequality}
Let $\mu_1,\sigma_1$ be measures on $\mathbb R^n$ and $\mu_2,\sigma_2$ be measures on $\mathbb R^m$. Given $0<p_1,p_2,q_1,q_2<\infty$  and suppose that $T$ is a bounded linear operator from $L^{p_2}(L^{p_1})(\mu_2\times \mu_1)$ to $L^{q_2}(L^{q_1})(\sigma_2\times \sigma_1)$. Then $T$ has an $\ell^2$-valued extension, that is, for all complex-valued functions $f_j$ in $L^{p_2}(L^{p_1})(\mu_2\times \mu_1)$ we have
\[
\Big\|\Big(\sum_{j}|T(f_j)|^2\Big)^{1/2}\Big\|_{L^{q_2}(L^{q_1})(\sigma_2\times \sigma_1)}\lesssim \Big\|\Big(\sum_{j}|f_j|^2\Big)^{1/2}\Big\|_{L^{p_2}(L^{p_1})(\mu_2\times \mu_1)}.
\]
\end{lemma}

\begin{proof}
By the Khintchine inequality,
\[
\Big(\sum_{j}|T(f_j)(x)|^2\Big)^{1/2} \lesssim \bigg(\int_{\Omega}\big| \sum_{j}\varepsilon_j(s) T(f_j)(x) \big|^{p_2}ds\bigg)^{1/p_2},
\]
where $(\varepsilon_j)$ is a Rademacher sequence defined on a probability space $\big(\Omega,\mathbb P\big)$. This gives us
\begin{align*}
&\Big\| \Big(\sum_{j}|T(f_j)|^2\Big)^{1/2}\Big\|_{L^{q_2}(L^{q_1})(\sigma_2\times \sigma_1)}\\
\lesssim &  \bigg(\int_{\mathbb R^m}\bigg(\int_{\mathbb R^n}\int_{\Omega} \Big| \sum_{j} \varepsilon_{j}(s) Tf_j(x_1,x_2)\Big|^{q_1} ds d\sigma_1(x_1) \bigg)^{q_2/q_1}d\sigma_2(x_2)\bigg)^{1/q_2}\\
= & \bigg(\int_{\mathbb R^m}\bigg(\int_{\Omega}\int_{\mathbb R^n} \Big| \sum_{j} \varepsilon_{j}(s) Tf_j(x_1,x_2)\Big|^{q_1} d\sigma_1(x_1) ds\bigg)^{q_2/q_1}d\sigma_2(x_2)\bigg)^{1/q_2}.\\
\end{align*}
By the Kahane-Khintchine inequality, we have
\begin{align*}
&\int_{\Omega}\int_{\mathbb R^n} \Big| \sum_{j} \varepsilon_{j}(s) Tf_j(x_1,x_2)\Big|^{q_1} d\sigma_1(x_1) ds\\
\lesssim & \bigg(\int_{\Omega}\bigg(\int_{\mathbb R^n} \Big| \sum_{j} \varepsilon_{j}(s) Tf_j(x_1,x_2)\Big|^{q_1} d\sigma_1(x_1) \bigg)^{q_2/q_1}ds\bigg)^{q_1/q_2}.
\end{align*}
Hence
\begin{align*}
& \bigg(\int_{\mathbb R^m}\bigg(\int_{\Omega}\int_{\mathbb R^n} \Big| \sum_{j} \varepsilon_{j}(s) Tf_j(x_1,x_2)\Big|^{q_1} d\sigma_1(x_1) ds\bigg)^{q_2/q_1}d\sigma_2(x_2)\bigg)^{1/q_2}\\
\lesssim & \bigg(\int_{\mathbb R^m} \int_{\Omega}\bigg(\int_{\mathbb R^n} \Big| \sum_{j} \varepsilon_{j}(s) Tf_j(x_1,x_2)\Big|^{q_1} d\sigma_1(x_1) \bigg)^{q_2/q_1}ds d\sigma_2(x_2)\bigg)^{1/q_2}\\
= & \bigg(\int_{\Omega} \int_{\mathbb R^m}\bigg(\int_{\mathbb R^n} \Big| \sum_{j} \varepsilon_{j}(s) Tf_j(x_1,x_2)\Big|^{q_1} d\sigma_1(x_1) \Big)^{q_2/q_1}d\sigma_2(x_2) ds\bigg)^{1/q_2}.
\end{align*}
Since $T$ is bounded from $L^{p_2}(L^{p_1})(\mu_2\times \mu_1)$ to $L^{q_2}(L^{q_1})(\sigma_2\times \sigma_1)$,
we see from the above arguments that
\begin{align*}
&\Big\|\Big(\sum_{j}|T(f_j)|^2\Big)^{1/2}\Big\|_{L^{q_2}(L^{q_1})(\sigma_2\times \sigma_1)}\\
\lesssim &\bigg(\int_{\Omega} \bigg( \int_{\mathbb R^m}\bigg(\int_{\mathbb R^n} \Big| \sum_{j} \varepsilon_{j}(s) f_j(x_1,x_2)\Big|^{p_1} d\mu_1(x_1) \bigg)^{p_2/p_1}d\mu_2(x_2)\bigg)^{q_2/p_2} ds\bigg)^{1/q_2}.
\end{align*}

When $q_2>p_2$, we see from  Minkowski's integral inequality that
\begin{align*}
&\bigg(\int_{\Omega} \bigg( \int_{\mathbb R^m}\bigg(\int_{\mathbb R^n} \Big| \sum_{j} \varepsilon_{j}(s) f_j(x_1,x_2)\Big|^{p_1} d\mu_1(x_1) \bigg)^{p_2/p_1}d\mu_2(x_2)\bigg)^{q_2/p_2} ds\bigg)^{1/q_2}\\
\leq & \bigg(\int_{\mathbb R^m} \bigg( \int_{\Omega}\bigg(\int_{\mathbb R^n} \Big| \sum_{j} \varepsilon_{j}(s) f_j(x_1,x_2)\Big|^{p_1} d\mu_1(x_1) \bigg)^{q_2/p_1}ds\bigg)^{p_2/q_2} d\mu_2(x_2)\bigg)^{1/p_2}.
\end{align*}
Using the Kahane-Khintchine inequality again, we get
\begin{align*}
& \bigg(\int_{\mathbb R^m} \bigg( \int_{\Omega}\bigg(\int_{\mathbb R^n} \Big| \sum_{j} \varepsilon_{j}(s) f_j(x_1,x_2)\Big|^{p_1} d\mu_1(x_1) \bigg)^{q_2/p_1}ds\bigg)^{p_2/q_2} d\mu_2(x_2)\bigg)^{1/p_2}\\
\lesssim & \bigg(\int_{\mathbb R^m} \bigg( \int_{\Omega}\Big(\int_{\mathbb R^n} \bigg| \sum_{j} \varepsilon_{j}(s) f_j(x_1,x_2)\Big|^{p_1} d\mu_1(x_1) \bigg)^{p_1/p_1}ds\bigg)^{p_2/p_1} d\mu_2(x_2)\bigg)^{1/p_2}\\
= & \bigg(\int_{\mathbb R^m} \bigg( \int_{\mathbb R^n}\bigg(\int_{\Omega} \Big| \sum_{j} \varepsilon_{j}(s) f_j(x_1,x_2)\Big|^{p_1} ds \Big)^{p_1/p_1}d\mu_1(x_1) \bigg)^{p_2/p_1} d\mu_2(x_2)\bigg)^{1/p_2}\\
\lesssim & \bigg(\int_{\mathbb R^m} \bigg( \int_{\mathbb R^n} \Big(\sum_{j}  |f_j(x_1,x_2)|^2 \Big)^{p_1/2}d\mu_1(x_1) \bigg)^{p_2/p_1} d\mu_2(x_2)\bigg)^{1/p_2}\\
= & \Big\|\Big(\sum_{j}|f_j|^2\Big)^{1/2}\Big\|_{L^{p_2}(L^{p_1})(\mu_2\times \mu_1)}.
\end{align*}
Hence for $q_2> p_2$,
\[
\Big\|\Big(\sum_{j}|T(f_j)|^2\Big)^{1/2}\Big\|_{L^{q_2}(L^{q_1})(\sigma_2\times \sigma_1)}\lesssim \Big\|\Big(\sum_{j}|f_j|^2\Big)^{1/2}\Big\|_{L^{p_2}(L^{p_1})(\mu_2\times \mu_1)}.
\]

When $q_2\leq p_2$, since $\big(\Omega,\mathbb P\big)$ is a probability space,   by H\"{o}lder's inequality, we get
\begin{align*}
&\bigg(\int_{\Omega} \Big( \int_{\mathbb R^m}\bigg(\int_{\mathbb R^n} \Big| \sum_{j} \varepsilon_{j}(s) f_j(x_1,x_2)\Big|^{p_1} d\mu_1(x_1) \bigg)^{p_2/p_1}d\mu_2(x_2)\bigg)^{q_2/p_2} ds\bigg)^{1/q_2}\\
\leq & \bigg(\int_{\Omega} \bigg( \int_{\mathbb R^m}\bigg(\int_{\mathbb R^n} \Big| \sum_{j} \varepsilon_{j}(s) f_j(x_1,x_2)\Big|^{p_1} d\mu_1(x_1) \bigg)^{p_2/p_1}d\mu_2(x_2)\bigg) ds\bigg)^{1/p_2}\\
= & \bigg(\int_{\mathbb R^m} \bigg( \int_{\Omega}\bigg(\int_{\mathbb R^n} \Big| \sum_{j} \varepsilon_{j}(s) f_j(x_1,x_2)\Big|^{p_1} d\mu_1(x_1) \bigg)^{p_2/p_1}ds\bigg)^{p_2/p_2} d\mu_2(x_2)\bigg)^{1/p_2}.
\end{align*}
Similar arguments as those for the case of $q_2>p_2$ show that when $q_2\leq p_2$,
\[
\Big\|\Big(\sum_{j}|T(f_j)|^2\Big)^{1/2}\Big\|_{L^{q_2}(L^{q_1})(\sigma_2\times \sigma_1)}\lesssim \Big\|\Big(\sum_{j}|f_j|^2\Big)^{1/2}\Big\|_{L^{p_2}(L^{p_1})(\mu_2\times \mu_1)}.
\]
This completes the proof.
\end{proof}

Next we give the weighted estimates of the partial fractional integral operators
$I_{\lambda_1}^1$ and $I_{\lambda_2}^2$.

\begin{lemma}\label{Lm:mixed frac}
Let $I_{\lambda_1}$ and $I_{\lambda_2}$ be fractional integral operators in $\mathbb R^n$ and $\mathbb R^m$, respectively,
where
$0<\lambda_1<n$ and $0<\lambda_2<m$.
Let
$1<p_1,q_1, p_2,q_2<\infty$ be constants such that $1/q_1+1/p_1'=\lambda_1/n$
and $1/q_2+1/p_2'=\lambda_2/m$.
Suppose that  $\mu_1\in A_{p_1,q_1}(\mathbb R^n)$ and $\mu_2\in A_{p_2,q_2}(\mathbb R^m)$.
Then for any measurable function $f$ defined on $\mathbb R^{n+m}$, we have
\begin{align}
\|I_{\lambda_1}^1f\|_{L^{p_2}(L^{q_1})(\mu_2^{p_2}\times \mu_1^{q_1})}
&\lesssim_{[\mu_1]_{A_{p_1,q_1}(\mathbb R^n)}} \|f\|_{L^{p_2}(L^{p_1})(\mu_2^{p_2}\times \mu_1^{p_1})},
   \label{eq:I:e1}\\
\|I_{\lambda_1}^1f\|_{L^{q_2}(L^{q_1})(\mu_2^{q_2}\times \mu_1^{q_1})}
&\lesssim_{[\mu_1]_{A_{p_1,q_1}(\mathbb R^n)}} \|f\|_{L^{q_2}(L^{p_1})(\mu_2^{q_2}\times \mu_1^{p_1})},
   \label{eq:I:e2}\\
\|I_{\lambda_2}^2f\|_{L^{q_2}(L^{p_1})(\mu_2^{q_2}\times \mu_1^{p_1})}
&
   \lesssim_{[\mu_2]_{A_{p_2,q_2}(\mathbb R^m)}}
   \|f\|_{L^{p_2}(L^{p_1})(\mu_2^{p_2}\times \mu_1^{p_1})},
   \label{eq:I:e3}\\
\|I_{\lambda_2}^2f\|_{L^{q_2}(L^{q_1})(\mu_2^{q_2}\times \mu_1^{q_1})}
&
\lesssim_{[\mu_2]_{A_{p_2,q_2}(\mathbb R^m)}}
   \|f\|_{L^{p_2}(L^{q_1})(\mu_2^{p_2}\times \mu_1^{q_1})}.
    \label{eq:I:e4}
\end{align}
\end{lemma}

\begin{proof} Observe that
\begin{align*}
\|I_{\lambda_1}^1f(\cdot,x_2)\|_{L^{q_1}(\mu_1^{q_1})} &= \|I_{\lambda_1}\big(f(\cdot,x_2)\big)\|_{L^{q_1}(\mu_1^{q_1})}\\
&\lesssim_{[\mu_1]_{A_{p_1,q_1}(\mathbb R^n)}}
     \|f(\cdot,x_2)\|_{L^{p_1}(\mu_1^{p_1})}.
\end{align*}
Taking the $L^{p_2}(\mu_2^{p_2})$ norm and the $L^{q_2}(\mu_2^{q_2})$ norm respectively, we get the first two inequalities.

On the other hand, by Minkowski's integral inequality, we get
\begin{align*}
& \|I_{\lambda_2}^2f\|_{L^{q_2}(L^{p_1})(\mu_2^{q_2}\times \mu_1^{p_1})} \\
&= \bigg(\int_{\mathbb R^m}\bigg(\int_{\mathbb R^n}\Big|\int_{\mathbb R^m}\frac{f(x_1,y_2)}{|x_2-y_2|^{\lambda_2}}dy_2\Big|^{p_1}\mu_1^{p_1}(x_1)
     dx_1\bigg)^{q_2/p_1}\mu_2^{q_2}(x_2)dx_2\bigg)^{1/q_2}\\
& \leq \bigg(\int_{\mathbb R^m}\bigg(\int_{\mathbb R^m}
   \bigg(\int_{\mathbb R^n}\frac{|f(x_1,y_2)|^{p_1}}{|x_2-y_2|^{\lambda_2 p_1}}\mu_1^{p_1}(x_1)dx_1\bigg)^{1/p_1}dy_2\bigg)^{q_2}
     \mu_2^{q_2}(x_2)dx_2\bigg)^{1/q_2}\\
& = \bigg(\int_{\mathbb R^m}\bigg(\int_{\mathbb R^m}
  \frac{\|f(\cdot,y_2)\|_{L^{p_1}(\mu_1^{p_1})}}
   {|x_2-y_2|^{\lambda_2}}dy_2\bigg)^{q_2}
     \mu_2^{q_2}(x_2)dx_2\bigg)^{1/q_2}\\
& \lesssim_{[\mu_2]_{A_{p_2,q_2}(\mathbb R^m)}}
  \bigg(\int_{\mathbb R^m}\|f(\cdot,x_2)
    \|_{L^{p_1}(\mu_1^{p_1})}^{p_2}
      \mu_2^{p_2}(x_2)dx_2\bigg)^{1/p_2}\\
& = \|f\|_{L^{p_2}(L^{p_1})(\mu_2^{p_2}\times
    \mu_1^{p_1})}.
\end{align*}
This proves (\ref{eq:I:e3}).
And (\ref{eq:I:e4}) can be proved similarly.
\end{proof}

Li,  Martikainen and Vuorinen~\cite{LiMartikainenVuorinen1806,LiMartikainenVuorinen1810}
showed that a product $bf$ can be expanded by paraproduct operators.
Specifically,
let $\Dw$ and $\mathscr D^{\beta}$ be dyadic systems in $\mathbb R^n$
 and  $\mathbb R^m$, respectively,
where $\omega=(\omega_j)_{j\in\mathbb{Z}}\in(\{0,1\}^n)^{\mathbb{Z}}$
and $\beta=(\beta_j)_{j\in\mathbb{Z}}\in (\{0,1\}^m)^{\mathbb{Z}}$.
The paraproduct operators are defined by
\[
A_1(b,f) = \sum_{Q\in\Dw,P\in\Db}\Delta_{Q\times P}b\Delta_{Q\times P}f,\quad A_2(b,f) = \sum_{Q\in\Dw,P\in\Db}\Delta_{Q\times P}bE_{Q}^1\Delta_{P}^2f,
\]
\[
A_3(b,f) = \sum_{Q\in\Dw,P\in\Db}\Delta_{Q\times P}b\Delta_{Q}^1E_{P}^2f,\quad A_4(b,f) = \sum_{Q\in\Dw,P\in\Db}\Delta_{Q\times P}b\ave{f}_{Q\times P},
\]
and
\[
A_5(b,f) = \sum_{Q\in\Dw,P\in\Db}E_{Q}^1\Delta_{P}^2b\Delta_{Q\times P}f,\quad A_6(b,f) = \sum_{Q\in\Dw,P\in\Db}E_{Q}^1\Delta_{P}^2b\Delta_{Q}^1E_{P}^2f,
\]
\[
A_7(b,f) = \sum_{Q\in\Dw,P\in\Db}\Delta_{Q}^1E_{P}^2b\Delta_{Q\times P}f,\quad A_8(b,f) = \sum_{Q\in\Dw,P\in\Db}\Delta_{Q}^1E_{P}^2bE_{Q}^1\Delta_{P}^2f.
\]
The "illegal" biparameter paraproduct is
\[
W(b,f)=\sum_{Q\in\Dw,P\in\Db} \ave{b}_{Q\times P}\Delta_{Q\times P}f.
\]
We can decompose the multiplication $bf$ in the biparameter sense,
\begin{equation}\label{eq:decomp}
bf = \sum_{k=1}^8 A_k(b,f)+W(b,f).
\end{equation}

In the following we give the weighted estimates of $A_k(b,\cdot)$ on mixed-norm spaces.

\begin{lemma}\label{Lm:paraproduct}
Suppose that $b\in \BMO_{\pro}(\nu)$, where $\nu := \mu_1^{1/p}\sigma_1^{-1/p}\otimes \mu_2^{1/q}\sigma_{2}^{-1/q}$, $\mu_1$, $\sigma_1\in A_{p}(\mathbb R^n)$,
$\mu_2$, $\sigma_2\in A_{q}(\mathbb R^m)$ and $1<p,q<\infty$.
Then for $k=1,2,3,4$, we have
\[
\|A_k(b,\cdot)\|_{L^{q}(L^p)(\mu_2\times \mu_1)\rightarrow L^{q}(L^{p})(\sigma_2\times \sigma_1)}\lesssim_{\substack{[\mu_1]_{A_{p}(\mathbb R^n)},[\mu_2]_{A_{q}(\mathbb R^m)}\\ [\sigma_1]_{A_{p}(\mathbb R^n)},[\sigma_2]_{A_{q}(\mathbb R^m)}}} \|b\|_{\BMO_{\pro}(\nu)}.
\]
\end{lemma}

The classical weighted $L^p$ norm  estimates of $A_k(b,\cdot)$  were
proved by Holmes, Petermichl and Wick in \cite{HolmesPetermichlWick}.
The main tool they used was the following weighted $H^1$-$\BMO_{\pro}$ duality estimate.

\begin{proposition}[{\cite[Proposition 4.1]{HolmesPetermichlWick}}]\label{prop:H1-BMO}
For every $b\in \BMO_{\pro}^{\Dw,\Db}(w)$, we have
\[
\pair{b}{\phi} \lesssim \|b\|_{\BMO_{\pro}^{\Dw,\Db}(w)} \|S_{\Dw,\Db}\phi\|_{L^1(w)}.
\]
\end{proposition}

By Proposition~\ref{prop:H1-BMO}, we have the following estimate
\begin{equation}\label{eq:s:e15}
\sum_{\substack{Q\in\Dw\\P\in\Db}}|\pair{b}{h_Q\otimes h_P}||c_{Q,P}|\lesssim \|b\|_{\BMO_{\pro}^{\Dw,\Db}(w)}\iint_{\mathbb R^{n+m}}\Big(\sum_{\substack{Q\in\Dw\\P\in\Db}}|c_{Q,P}|^2\frac{1_Q\otimes 1_P}{|Q||P|}\Big)^{1/2}w.
\end{equation}

Using the same idea as in \cite{HolmesPetermichlWick} we can prove   Lemma~\ref{Lm:paraproduct}.
Here we only outline the general strategy and leave the details to interested readers.

\begin{proof}[Proof of Lemma~\ref{Lm:paraproduct}]
By the dual property of mixed-norm spaces~\cite{BenedekPanzone}, it suffices to show that for any $f\in L^{q}(L^p)(\mu_2\times \mu_1)$ and $g\in L^{q'}(L^{p'})(\sigma_2'\times \sigma_1')$ with $\sigma_1'=\sigma_1^{1-p'}$ and $\sigma_2'=\sigma_2^{1-q'}$,
\[
|\pair{A_k(b,f)}{g}|\lesssim_{\substack{[\mu_1]_{A_{p}(\mathbb R^n)},[\mu_2]_{A_{q}(\mathbb R^m)}\\ [\sigma_1]_{A_{p}(\mathbb R^n)},[\sigma_2]_{A_{q}(\mathbb R^m)}}} \|b\|_{\BMO_{\pro}(\nu)}\|f\|_{L^{q}(L^p)(\mu_2\times \mu_1)}\|g\|_{L^{q'}(L^{p'})(\sigma_2'\times \sigma_1')}.
\]

First, we write $\pair{A_k(b,f)}{g}=\pair{b}{\phi}$, where $\phi$ depends on $f$ and $g$. By Proposition~\ref{prop:H1-BMO}, $|\pair{A_k(b,f)}{g}|\lesssim \|b\|_{\BMO_{\pro}^{\Dw,\Db}(\nu)}\|S_{\Dw,\Db}\phi\|_{L^1(\nu)}$.

Next we show that $S_{\Dw,\Db}\phi\lesssim (\mathcal{O}_1f)(\mathcal{O}_2g)$,
where $\mathcal{O}_1$ and $\mathcal{O}_2$ are operators that are combinations
of maximal operators $M_{S}$, $M_{\Dw}$, $M_{\Db}$ and square functions $S_{\Dw,\Db},S_{\Dw},S_{\Db}$. By Proposition \ref{pro:strong maximal} and
Lemma~\ref{Lm:square}, we have
\[
\|\mathcal{O}_1f\|_{L^{q}(L^p)(\mu_2\times \mu_1)}\lesssim_{[\mu_1]_{A_{p}(\mathbb R^n)},[\mu_2]_{A_{q}(\mathbb R^m)}} \|f\|_{L^{q}(L^p)(\mu_2\times \mu_1)},
\]
and
\[
\|\mathcal{O}_2g\|_{L^{q'}(L^{p'})(\sigma_2'\times \sigma_1')}\lesssim_{[\sigma_1]_{A_{p}(\mathbb R^n)},[\sigma_2]_{A_{q}(\mathbb R^m)}} \|g\|_{L^{q'}(L^{p'})(\sigma_2'\times \sigma_1')}.
\]

Finally, by H\"{o}lder's inequality, the $L^1(\nu)$-norm of $S_{\Dw,\Db}\phi$ is controlled by the $L^{q}(L^p)(\mu_2\times \mu_1)$ norm of $\mathcal{O}_1$ and the $L^{q'}(L^{p'})(\sigma_2'\times \sigma_1')$ norm of $\mathcal{O}_2$, i.e.,
\begin{align*}
\|S_{\Dw,\Db}\phi\|_{L^1(\nu)} & \lesssim  \|\mathcal{O}_1f\|_{L^{q}(L^p)(\mu_2\times \mu_1)}\|\mathcal{O}_2g\|_{L^{q'}(L^{p'})(\sigma_2'\times \sigma_1')}\\
& \lesssim_{\substack{[\mu_1]_{A_{p}(\mathbb R^n)},[\mu_2]_{A_{q}(\mathbb R^m)}\\ [\sigma_1]_{A_{p}(\mathbb R^n)},[\sigma_2]_{A_{q}(\mathbb R^m)}}} \|f\|_{L^{q}(L^p)(\mu_2\times \mu_1)}\|g\|_{L^{q'}(L^{p'})(\sigma_2'\times \sigma_1')}.
\end{align*}
Now we get the conclusion as desired.
\end{proof}

We are now ready to prove the main result.

\begin{proof}[Proof of Theorem \ref{thm:blo}]
By the representation formula in Theorem~\ref{thm:rep}, it suffices to prove the Bloom
type inequality for $\big[S_{\lambda_1,\omega}^{i,j,1},\big[b,S_{\lambda_2,\beta}^{s,t,2} \big]\big]$,
where $\omega=(\omega_j)_{j\in\mathbb{Z}}\in(\{0,1\}^n)^{\mathbb{Z}}$,
$\beta=(\beta_j)_{j\in\mathbb{Z}}\in (\{0,1\}^m)^{\mathbb{Z}}$,
$i,j,s,t\in \mathbb Z^{+}\cup \{0\}$ and
\[
S_{\lambda_1,\omega}^{i,j,1}f(x_1,x_2):=  S_{\lambda_1,\omega}^{i,j}\big(f(\cdot,x_2)\big)(x_1)
=\sum_{\substack{K,I,J\in\mathscr{D}^{\omega}\\ I^{(i)}=K, J^{(j)}=K}}a_{\lambda_1,IJK} h_{J}\otimes\pair{f}{h_I}_1,
\]
\[
S_{\lambda_2,\omega}^{s,t,2}f(x_1,x_2):=  S_{\lambda_2,\beta}^{s,t}\big(f(x_1,\cdot)\big)(x_2)
=\sum_{\substack{V,S,T\in\mathscr{D}^{\beta}\\ S^{(s)}=V, T^{(t)}=V}}a_{\lambda_2,STV} \pair{f}{h_S}_2\otimes h_{T},
\]
and the coefficients satisfy
\begin{equation}\label{eq:s:e11}
a_{\lambda_1,IJK}\leq \frac{|I|^{1/2}|J|^{1/2}}{|K|^{\lambda_1/n}} \quad
\text{and}\quad a_{\lambda_2,STV}\leq \frac{|S|^{1/2}|T|^{1/2}}{|V|^{\lambda_2/m}}.
\end{equation}
By \eqref{eq:decomp}, we have
\begin{align*}
\big[S_{\lambda_1,\omega}^{i,j,1},\big[b,S_{\lambda_2,\beta}^{s,t,2} \big]\big]f
=& S_{\lambda_1,\omega}^{i,j,1}\big(bS_{\lambda_2,\beta}^{s,t,2}f\big) - S_{\lambda_1,\omega}^{i,j,1} S_{\lambda_2,\beta}^{s,t,2}\big(bf\big) \\
& \qquad - bS_{\lambda_2,\beta}^{s,t,2}S_{\lambda_1,\omega}^{i,j,1}f + S_{\lambda_2,\beta}^{s,t,2}\big(bS_{\lambda_1,\omega}^{i,j,1}f\big)\\
=& E + \sum_{k=1}^{8}\Big[ S_{\lambda_1,\omega}^{i,j,1}\big(A_k\big(b,S_{\lambda_2,\beta}^{s,t,2}f\big)\big) - S_{\lambda_1,\omega}^{i,j,1} S_{\lambda_2,\beta}^{s,t,2}\big(A_k\big(b,f\big)\big)\\
 &\quad\quad\quad\quad - A_k\big(b,S_{\lambda_2,\beta}^{s,t,2}S_{\lambda_1,\omega}^{i,j,1}f\big) + S_{\lambda_2,\beta}^{s,t,2}\big(A_k\big(bS_{\lambda_1,\omega}^{i,j,1}f\big)\big)\Big],
\end{align*}
where
\[
E:=\sum_{\substack{K,I,J\in\mathscr{D}^{\omega}\\ I^{(i)}=K, J^{(j)}=K}} \sum_{\substack{V,S,T\in\mathscr{D}^{\beta}\\ S^{(s)}=V, T^{(t)}=V}} b_{IJST}a_{\lambda_1,IJK}a_{\lambda_2,STV}\pair{f}{h_I\otimes h_S}h_{J}\otimes h_{T}
\]
and
\[
b_{IJST}:=-\ave{b}_{I\times S}+\ave{b}_{I\times T}+\ave{b}_{J\times S}-\ave{b}_{J\times T}.
\]

We start looking at the sum over $k$.
First of all, combining \eqref{eq:major} and Lemmas~\ref{Lm:mixed frac} and \ref{Lm:paraproduct}, we  conclude that for $k\leq 4$ all the individual terms
are bounded.

Next we consider the case of $k\ge  5$. In this case the term $S_{\lambda_1,\omega}^{i,j,1}\big(A_k\big(b,S_{\lambda_2,\beta}^{s,t,2}f\big)\big)$ should be paired with another term.
For $k=5,6$, it is paired with $- A_k\big(b,S_{\lambda_2,\beta}^{s,t,2}S_{\lambda_1,\omega}^{i,j,1}f\big)$ and for $k=7,8$, it is  paired with $-S_{\lambda_1,\omega}^{i,j,1} S_{\lambda_2,\beta}^{s,t,2}\big(A_k\big(b,f\big)\big)$.
We consider only the case of $k=5$. Other cases can be proved similarly.

Since $S_{\lambda_2,\beta}^{s,t,2}S_{\lambda_1,\omega}^{i,j,1}
=S_{\lambda_1,\omega}^{i,j,1}S_{\lambda_2,\beta}^{s,t,2}$,
by \eqref{eq:major} and Lemma~\ref{Lm:mixed frac}, we have
\begin{equation}\label{eq:shift}
\|S_{\lambda_2,\beta}^{s,t,2}f\|_{L^{q_2}(L^{p_1})(\mu_2^{q_2}\times \mu_1^{p_1})}\lesssim_{[\mu_2]_{A_{p_2,q_2}(\mathbb R^m)}} \|f\|_{L^{p_2}(L^{p_1})(\mu_2^{p_2}\times \mu_1^{p_1})}.
    \end{equation}
By the dual property of mixed-norm spaces,
it suffices to
estimate
$\big|\bair{S_{\lambda_1,\omega}^{i,j,1}
\big(A_5\big(b,f\big)\big)$ $-A_5\big(b,S_{\lambda_1,\omega}^{i,j,1}f\big)}{g}\big|$
for $g\in L^{q_2'}(L^{q_1'})(\sigma_2^{-q_2'}\times \sigma_1^{-q_1'})$.

Since
\begin{align*}
& S_{\lambda_1,\omega}^{i,j,1}\big(A_5\big(b,f\big)\big)-A_5\big(b,S_{\lambda_1,\omega}^{i,j,1}f\big)\\
= & \sum_{\substack{K,I,J\in\Dw\\I^{(i)}=K,J^{(j)}=K}}\sum_{P\in\Db}a_{\lambda_1,IJK}[\ave{\pair{b}{h_{P}}_2}_{I}-\ave{\pair{b}{h_P}_2}_{J}]\pair{f}{h_I\otimes h_P}h_J\otimes h_Ph_P
\end{align*}
and
\begin{align*}
\ave{\pair{b}{h_{P}}_2}_{I}-\ave{\pair{b}{h_P}_2}_{J} =& \ave{\pair{b}{h_{P}}_2}_{I}- \ave{\pair{b}{h_{P}}_2}_{K}+\ave{\pair{b}{h_{P}}_2}_{K}- \ave{\pair{b}{h_P}_2}_{J}\\
= & \sum_{r=1}^{i}\ave{\Delta_{I^{(r)}}\pair{b}{h_{P}}_2}_{I}-\sum_{l=1}^{j}
\ave{\Delta_{J^{(l)}}\pair{b}{h_{P}}_2}_{J},
\end{align*}
  we reduce the problem to estimate
\begin{equation}\label{eq:a}
\sum_{r=1}^{i}\sum_{\substack{K,I,J\in\Dw\\I^{(i)}=K,J^{(j)}=K}}
\sum_{P\in\Db}|a_{\lambda_1,IJK}||I^{(r)}|^{-1/2}|\pair{b}{h_{I^{(r)}}\otimes h_P}||\pair{f}{h_I\otimes h_P}|\ave{|\pair{g}{h_J}_1|}_P
\end{equation}
and
\begin{equation}\label{eq:b}
\sum_{l=1}^{j}\sum_{\substack{K,I,J\in\Dw\\I^{(i)}=K,J^{(j)}=K}}\sum_{P\in\Db}|a_{\lambda_1,IJK}||J^{(l)}|^{-1/2}|\pair{b}{h_{J^{(l)}}\otimes h_P}||\pair{f}{h_I\otimes h_P}|\ave{|\pair{g}{h_J}_1|}_P.
\end{equation}

First, we estimate (\ref{eq:a}). For $1\le r\le  i$, we have
\begin{align*}
&\sum_{\substack{K,I,J\in\Dw\\I^{(i)}=K,J^{(j)}=K}}\sum_{P\in\Db}|a_{\lambda_1,IJK}||I^{(r)}|^{-1/2}|\pair{b}{h_{I^{(r)}}\otimes h_P}||\pair{f}{h_I\otimes h_P}|\ave{|\pair{g}{h_J}_1|}_P\\
=& \sum_{K\in\Dw}\sum_{\substack{Q\in\Dw\\ Q^{(i-r)}=K\\P\in\Db}}\sum_{\substack{I,J\in\Dw\\ I^{(r)}=Q\\J^{(j)}=K}}
|a_{\lambda_1,IJK}||Q|^{-1/2}|\pair{b}{h_{Q}\otimes h_P}||\pair{f}{h_I\otimes h_P}|\ave{|\pair{g}{h_J}_1|}_P\\
\leq & \sum_{K\in\Dw}\sum_{\substack{Q\in\Dw\\ Q^{(i-r)}=K\\P\in\Db}} |\pair{b}{h_{Q}\otimes h_P}||Q|^{1/2}|P|^{1/2}|K|^{1-\lambda_1}\ave{|\Delta_{K\times P}^{i,0}f|}_{Q\times P} \ave{|g|}_{K\times P},
\end{align*}
where we use (\ref{eq:s:e11}) in the last step.
We see from (\ref{eq:s:e15}) that
\begin{align*}
&\sum_{K\in\Dw}\sum_{\substack{Q\in\Dw\\ Q^{(i-r)}=K\\P\in\Db}} |\pair{b}{h_{Q}\otimes h_P}||Q|^{1/2}|P|^{1/2}|K|^{1-\lambda_1}\ave{|\Delta_{K\times P}^{i,0}f|}_{Q\times P} \ave{|g|}_{K\times P}\\
\lesssim &\|b\|_{\BMO_{\pro}(\nu)}\iint_{\mathbb R^{n+m}}\Big(\sum_{\substack{K\in \Dw\\P\in \Db}}\big[M_{\Dw,\Db}\Delta_{K\times P}^{i,0}f]^2\Big)^{1/2}M_{\lambda_1,\Dw}^1\Big(M_{\Db}^2g\Big)\cdot \nu,
\end{align*}
where $\nu = \mu_1\sigma_1^{-1}\otimes \mu_2\sigma_2^{-1}$. By H\"{o}lder's inequality,
\begin{align*}
&\iint_{\mathbb R^{n+m}}\Big(\sum_{\substack{K\in \Dw\\P\in \Db}}\big[M_{\Dw,\Db}\Delta_{K\times P}^{i,0}f]^2\Big)^{1/2}M_{\lambda_1,\Dw}^1\Big(M_{\Db}^2g\Big)\cdot \nu\\
\leq & \Big\|\Big(\sum_{\substack{K\in \Dw\\P\in \Db}}\big[M_{\Dw,\Db}\Delta_{K\times P}^{i,0}f]^2\Big)^{1/2}\Big\|_{L^{q_2}(L^{p_1})(\mu_2^{q_2}\times \mu_1^{p_1})}\\
&\qquad \cdot
\Big\|M_{\lambda_1,\Dw}^1\Big(M_{\Db}^2g\Big) \Big\|_{L^{q_2'}(L^{p_1'})(\sigma_2^{-q_2'}\times \sigma_1^{-p_1'})}.
\end{align*}
Since $\mu_1\in A_{p_1,q_1}(\mathbb R^n)$ implies $\mu_1^{p_1}\in A_{p_1}(\mathbb R^n)$ and $\mu_2\in A_{p_2,q_2}(\mathbb R^m)$ implies
$\mu_2^{q_2}\in A_{q_2}(\mathbb R^m)$, by Proposition~\ref{pro:strong maximal} and Lemma~\ref{Lm:square},
\begin{align*}
&\Big\|\Big(\sum_{\substack{K\in \Dw\\P\in \Db}}\big[M_{\Dw,\Db}\Delta_{K\times P}^{i,0}f]^2\Big)^{1/2}\Big\|_{L^{q_2}(L^{p_1})(\mu_2^{q_2}\times \mu_1^{p_1})}
  \\
&\lesssim_{[\mu_1]_{ A_{p_1,q_1}(\mathbb R^n)},[\mu_2]_{A_{p_2,q_2}(\mathbb R^m)}} \|f\|_{L^{q_2}(L^{p_1})(\mu_2^{q_2}\times \mu_1^{p_1})}.
\end{align*}
By (\ref{eq:s:e5}), we have
\begin{equation}\label{eq:s:e12}
M_{\lambda_1,\Dw}^1\Big(M_{\Db}^2g\Big) \lesssim I_{\lambda_1,\Dw}^1\Big(M_{\Db}^2g\Big).
\end{equation}
Since  $\sigma_1\in A_{p_1,q_1}(\mathbb R^n)$, we have $\sigma_1^{-1}\in A_{q_1',p_1'}$.
It follows from Lemma~\ref{Lm:mixed frac} that
\[
\Big\|M_{\lambda_1,\Dw}^1\Big(M_{\Db}^2g\Big) \Big\|_{L^{q_2'}(L^{p_1'})(\sigma^{-q_2'}_2\times \sigma_1^{-p_1'})}\lesssim_{[\sigma_1]_{ A_{p_1,q_1}(\mathbb R^n)}} \Big\|M_{\Db}^2g \Big\|_{L^{q_2'}(L^{q_1'})(\sigma^{-q_2'}_2\times \sigma_1^{-q_1'})}.
\]

On the other hand, we see from $\sigma_2\in A_{p_2,q_2}(\mathbb R^m)$
that $\sigma^{-p_2'}_2\in A_{p_2'}(\mathbb R^m)$. By
Proposition~\ref{pro:strong maximal}, we get
\[
\Big\|M_{\Db}^2g \Big\|_{L^{q_2'}(L^{q_1'})(\sigma^{-q_2'}_2\times \sigma_1^{-q_1'})}\lesssim_{[\sigma_1]_{ A_{p_1,q_1}(\mathbb R^n)},[\sigma_2]_{A_{p_2,q_2}(\mathbb R^m)}} \|g\|_{L^{q_2'}(L^{q_1'})(\sigma^{-q_2'}_2\times \sigma_1^{-q_1'})}.
\]
So we conclude that
\begin{align*}
\quad &\sum_{\substack{K,I,J\in\Dw\\I^{(i)}=K,J^{(j)}=K}}\sum_{P\in\Db}|a_{\lambda_1,IJK}||I^{(r)}|^{-1/2}|\pair{b}{h_{I^{(r)}}\otimes h_P}||\pair{f}{h_I\otimes h_P}|\ave{|\pair{g}{h_J}_1|}_P\\
&\!\!\lesssim_{\substack{[\mu_1]_{A_{p_1,q_1}(\mathbb R^n)},[\mu_2]_{A_{p_2,q_2}(\mathbb R^m)} \\ [\sigma_1]_{A_{p_1,q_1}(\mathbb R^n)},[\sigma_2]_{A_{p_2,q_2}(\mathbb R^m)}}}
 \!\!
 \|b\|_{\BMO_{\pro}(\nu)} \|f\|_{L^{q_2}(L^{p_1})(\mu_2^{q_2}\times \mu_1^{p_1})}
 \|g\|_{L^{q_2'}(L^{q_1'})(\sigma^{-q_2'}_2\times \sigma_1^{-q_1'})}.
\end{align*}

Next, we estimate (\ref{eq:b}).
Using the same method as that for estimating (\ref{eq:a}),
we have
\begin{align*}
&\sum_{\substack{K,I,J\in\Dw\\I^{(i)}=K,J^{(j)}=K}}\sum_{P\in\Db}|a_{\lambda_1,IJK}||J^{(l)}|^{-1/2}|\pair{b}{h_{J^{(l)}}\otimes h_P}||\pair{f}{h_I\otimes h_P}|\ave{|\pair{g}{h_J}_1|}_P\\
=& \sum_{K\in\Dw}\sum_{\substack{Q\in\Dw\\ Q^{(j-l)}=K\\P\in\Db}}\sum_{\substack{I,J\in\Dw\\ I^{(i)}=K\\J^{(l)}=Q}}
|a_{\lambda_1,IJK}||Q|^{-1/2}|\pair{b}{h_{Q}\otimes h_P}||\pair{f}{h_I\otimes h_P}|\ave{|\pair{g}{h_J}_1|}_P\\
\leq & \sum_{K\in\Dw}\sum_{\substack{Q\in\Dw\\ Q^{(j-l)}=K\\P\in\Db}} |\pair{b}{h_{Q}\otimes h_P}||Q|^{1/2}|P|^{1/2}|K|^{1-\lambda_1}\ave{|\Delta_{K\times P}^{i,0}f|}_{K\times P} \ave{|g|}_{Q\times P}\\
\lesssim &\|b\|_{\BMO_{\pro}(\nu)}\iint_{\mathbb R^{n+m}}\Big(\sum_{\substack{K\in \Dw\\P\in \Db}}\big[M_{\lambda_1,\Dw}^1\big(M_{\Db}^2\Delta_{K\times P}^{i,0}f\big)]^2\Big)^{1/2}M_{\Dw,\Db}g\cdot \nu\\
\leq & \|b\|_{\BMO_{\pro}(\nu)} \Big\|\Big(\sum_{\substack{K\in \Dw\\P\in \Db}}\big[M_{\lambda_1,\Dw}^1\big(M_{\Db}^2\Delta_{K\times P}^{i,0}f\big)]^2\Big)^{1/2}\Big\|_{L^{q_2}(L^{q_1})(\mu_2^{q_2}\times \mu_1^{q_1})}\\
&\qquad \cdot
\Big\|M_{\Dw,\Db}g \Big\|_{L^{q_2'}(L^{q_1'})(\sigma^{-q_2'}_2
\times \sigma_1^{-q_1'})}.
\end{align*}
Also by (\ref{eq:s:e5}) we get
\[
M_{\lambda_1,\Dw}^1\big(M_{\Db}^2\Delta_{K\times P}^{i,0}f\big)\lesssim I_{\lambda_1,\Dw}^1\big(M_{\Db}^2\Delta_{K\times P}^{i,0}f\big).
\]
By Lemma~\ref{Lm:mixed frac}, $I_{\lambda_1,\Dw}^1$ is bounded from $L^{q_2}(L^{p_1})(\mu_2^{q_2}\times \mu_1^{p_1})$ to $L^{q_2}(L^{q_1})(\mu_2^{q_2}\times \mu_1^{q_1})$.
Note that $I_{\lambda_1,\Dw}^1$ is linear.
It follows from Lemma~\ref{Lm:vector value inequality} that
\begin{align*}
&\Big\|\Big(\sum_{\substack{K\in \Dw\\P\in \Db}}\big[M_{\lambda_1,\Dw}^1\big(M_{\Db}^2\Delta_{K\times P}^{i,0}f\big)]^2\Big)^{1/2}\Big\|_{L^{q_2}(L^{q_1})(\mu_2^{q_2}\times \mu_1^{q_1})}\\
\lesssim & \Big\|\Big(\sum_{\substack{K\in \Dw\\P\in \Db}}\big[I_{\lambda_1,\Dw}^1\big(M_{\Db}^2\Delta_{K\times P}^{i,0}f\big)]^2\Big)^{1/2}\Big\|_{L^{q_2}(L^{q_1})(\mu_2^{q_2}\times \mu_1^{q_1})}\\
\lesssim&_{[\mu_1]_{A_{p_1,q_1}}(\mathbb R^n)}\Big\|\Big(\sum_{\substack{K\in \Dw\\P\in \Db}}\big[M_{\Db}^2\Delta_{K\times P}^{i,0}f]^2\Big)^{1/2}\Big\|_{L^{q_2}(L^{p_1})(\mu_2^{q_2}\times \mu_1^{p_1})}.
\end{align*}
Now we see from  Proposition \ref{pro:strong maximal} and Lemma~\ref{Lm:square}
that
\begin{align*}
&\Big\|\Big(\sum_{\substack{K\in \Dw\\P\in \Db}}\big[M_{\Db}^2\Delta_{K\times P}^{i,0}f]^2\Big)^{1/2}\Big\|_{L^{q_2}(L^{p_1})(\mu_2^{q_2}
\times \mu_1^{p_1})} \\
&\lesssim_{[\mu_1]_{A_{p_1,q_1}(\mathbb R^n)},[\mu_2]_{A_{p_2,q_2}(\mathbb R^m)}}
\|f\|_{L^{q_2}(L^{p_1})(\mu_2^{q_2}\times \mu_1^{p_1})}.
\end{align*}
By Proposition \ref{pro:strong maximal},
\[
\Big\|M_{\Dw,\Db}g \Big\|_{L^{q_2'}(L^{q_1'})(\sigma^{-q_2'}_2\times \sigma_1^{-q_1'})}\lesssim_{[\sigma_1]_{A_{p_1,q_1}(\mathbb R^n)},[\sigma_2]_{A_{p_2,q_2}(\mathbb R^m)}} \|g\|_{L^{q_2'}(L^{q_1'})(\sigma^{-q_2'}_2\times \sigma_1^{-q_1'})}.
\]
It follows that
\begin{align*}
\quad &\sum_{\substack{K,I,J\in\Dw\\I^{(i)}=K,J^{(j)}=K}}\sum_{P\in\Db}|a_{\lambda_1,IJK}||J^{(l)}|^{-1/2}|\pair{b}{h_{J^{(l)}}\otimes h_P}||\pair{f}{h_I\otimes h_P}|\ave{|\pair{g}{h_J}_1|}_P\\
&\!\!\lesssim_{\substack{[\mu_1]_{A_{p_1,q_1}(\mathbb R^n)},[\mu_2]_{A_{p_2,q_2}(\mathbb R^m)} \\
  [\sigma_1]_{A_{p_1,q_1}(\mathbb R^n)},[\sigma_2]_{A_{p_2,q_2}(\mathbb R^m)}}}
\!\!\|b\|_{\BMO_{\pro}(\nu)} \|f\|_{L^{q_2}(L^{p_1})(\mu_2^{q_2}\times \mu_1^{p_1})}
    \|g\|_{L^{q_2'}(L^{q_1'})(\sigma^{-q_2'}_2\times \sigma_1^{-q_1'})}.
\end{align*}
Taking the summation   over $r$ and $l$, we obtain
\begin{align*}
&\hskip -2em \big|\bair{S_{\lambda_1,\omega}^{i,j,1}\big(A_5\big(b,f\big)\big)-A_5\big(b,S_{\lambda_1,\omega}^{i,j,1}f\big)}{g}\big|\\
\lesssim
&_{\substack{[\mu_1]_{A_{p_1,q_1}(\mathbb R^n)},[\mu_2]_{A_{p_2,q_2}(\mathbb R^m)} \\ [\sigma_1]_{A_{p_1,q_1}(\mathbb R^n)},[\sigma_2]_{A_{p_2,q_2}(\mathbb R^m)}}}(1+\max(i,j))\|b\|_{\BMO_{\pro}(\nu)} \\
&\qquad \cdot \|f\|_{L^{q_2}(L^{p_1})(\mu_2^{q_2}\times \mu_1^{p_1})} \|g\|_{L^{q_2'}(L^{q_1'})(\sigma^{-q_2'}_2\times \sigma_1^{-q_1'})}.
\end{align*}
Combining with \eqref{eq:shift}, we get
\begin{align*}
&\big\|S_{\lambda_1,\omega}^{i,j,1}\big(A_5\big(b,S_{\lambda_2,\beta}^{s,t,2}f\big)\big)-A_5\big(b,S_{\lambda_2,\beta}^{s,t,2}
S_{\lambda_1,\omega}^{i,j,1}f\big)\big\|_{L^{q_2}(L^{q_1})(\sigma^{q_2}_2,\times \sigma_1^{q_1})}\\
\lesssim &_{\substack{[\mu_1]_{A_{p_1,q_1}(\mathbb R^n)},[\mu_1]_{A_{p_2,q_2}(\mathbb R^m)} \\ [\sigma_1]_{A_{p_1,q_1}(\mathbb R^n)},[\sigma_2]_{A_{p_2,q_2}(\mathbb R^m)}}}(1+\max(i,j))\|b\|_{\BMO_{\pro}(\nu)} \|f\|_{L^{p_2}(L^{p_1})(\mu_2^{p_2}\times \mu_1^{p_1})}.
\end{align*}
It remains to estimate
\[
E:=\sum_{\substack{K,I,J\in\mathscr{D}^{\omega}\\ I^{(i)}=K, J^{(j)}=K}} \sum_{\substack{V,S,T\in\mathscr{D}^{\beta}\\ S^{(s)}=V, T^{(t)}=V}} b_{IJST}a_{\lambda_1,IJK}a_{\lambda_2,STV}\pair{f}{h_I\otimes h_S}h_{J}\otimes h_{T},
\]
where
\[
b_{IJST}:=-\ave{b}_{I\times S}+\ave{b}_{I\times T}+\ave{b}_{J\times S}-\ave{b}_{J\times T}.
\]
Noting that
\begin{align*}
b_{IJST}=&-\sum_{r=1}^{i}\sum_{l=1}^{s}\bve{\Delta_{I^{(r)}\times S^{(l)}}b}_{I\times S}+\sum_{r=1}^{i}\sum_{l=1}^{t}\bve{\Delta_{I^{(r)}\times T^{(l)}}b}_{I\times T}\\
&+\sum_{r=1}^{j}\sum_{l=1}^{s}\bve{\Delta_{J^{(r)}\times S^{(l)}}b}_{J\times S}-\sum_{r=1}^{j}\sum_{l=1}^{t}\bve{\Delta_{J^{(r)}\times T^{(l)}}b}_{J\times T},
\end{align*}
with similar arguments we get that
\begin{align*}
& \big\|\big[S_{\lambda_1,\omega}^{i,j,1},\big[b,S_{\lambda_2,\beta}^{s,t,2} \big]\big]\big\|_{L^{p_2}(L^{p_1})(\mu_2^{p_2}\times \mu_1^{p_1})\rightarrow L^{q_2}(L^{q_1})(\sigma^{q_2}_2,\times \sigma_1^{q_1})}\\
\lesssim &_{\substack{[\mu_1]_{A_{p_1,q_1}(\mathbb R^n)},[\mu_2]_{A_{p_2,q_2}(\mathbb R^m)} \\ [\sigma_1]_{A_{p_1,q_1}(\mathbb R^n)},[\sigma_2]_{A_{p_2,q_2}(\mathbb R^m)}}}(1+\max(i,j))(1+\max(s,t))\|b\|_{\BMO_{\pro}(\nu)}.
\end{align*}
This completes the proof.
\end{proof}

\end{document}